\documentclass[a4paper,11pt]{article} 

\usepackage{hyperref}

\usepackage[a4paper,margin=2.7cm]{geometry}
\usepackage{titlesec}
\titleformat{\section}{\LARGE\center\bfseries\scshape}{\thesection.}{.7em}{}
\titlespacing*{\section}{0pt}{3.5ex plus 0ex minus 0ex}{1.5ex plus 0ex}
\titleformat{\subsection}{\Large\center\bfseries\scshape}{\thesubsection.}{.7em}{}
\titlespacing*{\subsection}{0pt}{3.5ex plus 0ex minus 0ex}{1.5ex plus 0ex}

\makeatletter
\g@addto@macro\normalsize{%
  \setlength\abovedisplayskip{10pt}
  \setlength\belowdisplayskip{10pt}	
  \setlength\abovedisplayshortskip{10pt}
  \setlength\belowdisplayshortskip{10pt}}
\makeatother

\usepackage[english]{babel}
\usepackage[T1]{fontenc}
\usepackage{amsfonts}
\usepackage{mathrsfs}
\usepackage{amssymb}
\usepackage{amsmath}

\usepackage{changebar}
\usepackage{enumitem}
\setlist{nolistsep}
\usepackage{amsthm}
\usepackage[capitalize]{cleveref}
\usepackage{chngcntr}
\usepackage[normalem]{ulem}

\newtheoremstyle{plain}{1mm}{2mm}{\slshape}{}{\color{Blue}\bfseries}{.}{.5em}{}
\newtheoremstyle{definition}{1mm}{2mm}{}{}{\color{Blue}\bfseries}{.}{.5em}{}
\theoremstyle{plain}
\newtheorem*{Thm}{Theorem}	
\newtheorem{Theorem}{Theorem}[section]

\newtheorem{Lemma}[Theorem]{Lemma}
\newtheorem{Proposition}[Theorem]{Proposition}
\newtheorem{Corollary}[Theorem]{Corollary}

\theoremstyle{definition}
\newtheorem{Definition}[Theorem]{Definition}
\newtheorem{Remark}[Theorem]{Remark}
\newtheorem{Example}[Theorem]{Example}

\theoremstyle{plain} 
\newcounter{MainTheoremCounter}

\newtheorem{Maintheorem}[MainTheoremCounter]{Theorem}
\newtheorem{Maincorollary}[MainTheoremCounter]{Corollary}

\theoremstyle{plain}
\newtheorem*{namedthm}{\namedthmname}
\newcounter{namedthm}
	

\usepackage{xcolor}

\definecolor{Scarlet}{rgb}{0.65, 0.10, 0.0}
\definecolor{Blue}{rgb}{0.0, 0.05, 0.39}
\definecolor{Green}{rgb}{0.3, 0.6 ,0.2}

\hypersetup{	citecolor = Scarlet,colorlinks,
			linkcolor = black,
			urlcolor = Scarlet}

\newcommand{\N}{\mathbb{N}}
\newcommand{\Z}{\mathbb{Z}}
\newcommand{\R}{\mathbb{R}}
\newcommand{\C}{\mathbb{C}}
\newcommand{\Q}{\mathbb{Q}}

\newcommand{\define}[1]{\emph{#1}}

\renewcommand{\epsilon}{\varepsilon}
\renewcommand{\leq}{\leqslant}
\renewcommand{\geq}{\geqslant}
\renewcommand{\setminus}{{\backslash}}
\renewcommand{\Re}{{\rm Re}}

\renewcommand{\colon}{\nobreak\mskip2mu\mathpunct{}\nonscript\mkern-\thinmuskip{:}\mskip6muplus1mu\relax}

\newcommand{\xbm}{(X,\mathcal{B},\mu)}

\newcommand{\xbmt}{(X,\mathcal{B},\mu,T)}

\newcommand{\mob}{\boldsymbol{\mu}}
\newcommand{\lio}{\boldsymbol{\lambda}}
\newcommand{\tot}{\boldsymbol{\varphi}}

\newcommand{\vep}{\varepsilon}

\newcommand{\1}{1}
\newcommand{\cm}{\mathcal{M}}

\renewcommand{\P}{\mathbb{P}}
\newcommand{\T}{\mathbb{T}}

\newcommand{\G}{\mathbf{G}}
\newcommand{\mconv}{\cm_{0}}
\newcommand{\Oh}{{\rm O}}
\newcommand{\oh}{{\rm o}}
\newcommand{\ov}[1]{\overline{#1}}
\newcommand{\acc}[1]{{\rm acc}{#1}}
\newcommand{\diam}[1]{{\rm diam}\left({#1}\right)}

\newcommand{\Dzero}{\mathcal{D}}

\newcommand{\Drat}{\mathcal{D}_{\text{\normalfont rat}}}

\newcommand{\ME}{\mathcal{E}}
\newcommand{\MO}{\mathcal{O}}

\begin{document}
\allowdisplaybreaks

\title{A Structure Theorem for Level Sets of Multiplicative Functions and Applications}
\author{V.\ Bergelson\thanks{The first author gratefully acknowledges the support of the NSF under grant DMS-1500575.} \and J.\ Ku\l aga-Przymus\thanks{Research supported by Narodowe Centrum Nauki UMO-2014/15/B/ST1/03736 and the European Research Council (ERC) under the European Union's Horizon 2020 research and innovation programme (grant agreement No 647133 (ICHAOS)).} \and M.\ Lema\'nczyk\thanks{Research supported by Narodowe Centrum Nauki UMO-2014/15/B/ST1/03736 and the EU grant ``AOS'', FP7-PEOPLE-2012-IRSES, No 318910.} \and F.\ K.\ Richter}

\date{\small \today}

\maketitle
\vspace{-0.3 cm}
\begin{abstract}
\noindent
Given a level set $E$ of an arbitrary multiplicative function $f$, we establish, by building on the fundamental work of Frantzikinakis and Host \cite{FH17-2,FH17}, a structure theorem which gives a decomposition of $\1_E$ into an \define{almost periodic} and a \define{pseudo-random} part.
Using this structure theorem together with the technique developed by the authors in \cite{BKLR16arXiv}, we obtain the following result pertaining to polynomial multiple recurrence.
\begin{Thm}
Let $E=\{n_1<n_2<\ldots\}$ be a level set of an arbitrary multiplicative function with positive density. Then the following are equivalent:
\begin{itemize}
\item
$E$ is divisible, i.e.\ the upper density of the set $E\cap u\N$ is positive for all $u\in\N$;
\item
$E$  is an averaging set of polynomial multiple recurrence, i.e.
for all measure preserving systems $\xbmt$, all $A\in\mathcal{B}$ with
$\mu(A)>0$, all $\ell\geq 1$ and all polynomials
$p_i\in\Z[x]$, $i=1,\ldots,\ell$, with $p_i(0)=0$ we have
\begin{equation*}
\lim_{N\to\infty}\frac{1}{N}\sum_{j=1}^N
\mu\big(A\cap T^{-p_1(n_j)}A\cap\ldots\cap T^{-p_\ell(n_j)}A\big)>0.
\end{equation*}
\end{itemize}
\end{Thm}
We also show that if a level set $E$ of a multiplicative function has positive upper density, then any self-shift $E-r$, $r\in E$, is a set of averaging 
polynomial multiple recurrence. This in turn leads to the following refinement of the polynomial Szemer{\'e}di theorem (cf. \cite{BL96}).

\begin{Thm}
Let $E$ be a level set of an arbitrary multiplicative function, suppose $E$ has positive upper density and let $r\in E$.
Then for any set $D\subset \N$ with positive upper density and any polynomials $p_i\in\Q[t]$, $i=1,\ldots,\ell$, which satisfy
$p_i(\Z)\subset\Z$ and $p_i(0)=0$ for all $i\in\{1,\ldots,\ell\}$,
there exists $\beta>0$ such that the set
$$
\left\{\,n\in E-r:\overline{d}\Big(
D\cap (D-p_1(n))\cap \ldots\cap(D-p_\ell(n))
\Big)>\beta \,\right\}
$$
has positive lower density.
\end{Thm}
\end{abstract}
\small
\tableofcontents
\thispagestyle{empty}
\normalsize

\section{Introduction}
\label{sec:intro}

In this paper we utilize facts, techniques and ideology coming from multiplicative number theory to obtain refinements and enhancements of some classical results in the theory of multiple recurrence.

An arithmetic function $f\colon\N=\{1,2,\ldots,\}\to\C$ is called \define{multiplicative} if $f(1)=1$ and $f(m n)=f(m)\cdot f(n)$ for all relatively prime $m,n\in\N$ and it is called \define{completely multiplicative} if $f(1)=1$ and $f(m n)=f(m)\cdot f(n)$ for all $m,n\in\N$. Let $\cm$ denote the set of all multiplicative functions bounded in modulus by $1$.
To motivate our results, we start by formulating a dichotomy theorem for the class $\mconv$ of all multiplicative functions $f\in\cm$ with the property that $\lim_{N\to\infty}\frac{1}{N}\sum_{n=1}^N f(qn+r)$ exists for all $q,r\in\N$.
This dichotomy theorem for $\mconv$ follows quickly by combining the work of Daboussi and Delange \cite{DD74,DD82,Delange72}, Frantzikinakis and Host \cite{FH17} and Bellow and Losert \cite{BL85} and is closely related to \cite[Theorem 1.1]{FH17}. It asserts that any function in $\mconv$ is either `highly structured' or exhibits (pseudo-)random behavior.
To formulate this theorem we first have to introduce \define{Besicovitch rationally almost periodic} functions, which epitomize `structure', and Gowers' notion of \define{uniformity}, which epitomizes (pseudo-)randomness.

The \define{Besicovitch seminorm} $\|.\|_B$ for a bounded function $f\colon \N\rightarrow \C$ is
defined as 
\begin{equation}
\label{eq:def-besic-seminorm}
\|f\|_B:=\limsup_{N\to\infty}\frac1N\sum_{n=1}^N|f(n)|.
\end{equation}

\begin{Definition}[\cite{Besicovitch55,BL85,BKLR16arXiv}]\label{def:d3}
Let $f\colon \N\to \C$ be a bounded arithmetic function.
\begin{itemize}
\item
$f$ is called \define{Besicovitch almost periodic} if for every $\vep>0$ there exists a
trigonometric polynomial $P(n)=\sum_{j=1}^k c_je(n\theta_j)$
with $c_1,\ldots,c_k\in\C$ and $\theta_1,\ldots,\theta_k\in\R$ such that $\|f-P\|_B<\vep$.
\item
Following \cite{BKLR16arXiv}, we call $f\colon \N\to\C$ \define{Besicovitch rationally almost periodic} if for every $\vep>0$ there
exists a periodic function $P\colon \N\to\C$ (or, equivalently, a
trigonometric polynomial $P(n)=\sum_{j=1}^k c_je(n\theta_j)$
with $c_1,\ldots,c_k\in\C$ and $\theta_1,\ldots,\theta_k\in\Q$) such that $\|f-P\|_B<\vep$.
\end{itemize}
\end{Definition}

Many multiplicative functions are Besicovitch rationally almost periodic. For instance, any bounded multiplicative function taking values in $[0,\infty)$ is Besicovitch rationally almost periodic (as is shown in \cref{ren:0-1-valued-RAP} below). 

Next, we give the definition of the \define{Gowers uniformity seminorms} and of \define{uniform functions}.
These notions play a central role in additive combinatorics and have useful applications to ergodic theory.

Given $N\in\N$ we write $[N]$ for the interval $\{1,2,\ldots,N\}$.

\begin{Definition}[Gowers uniformity seminorms, \cite{Gowers01,GTZ12}]\label{def:GN}
Let $N\in\N$ and let $\Z/N\Z$ denote the finite cyclic group with $N$ elements. For $h\in\Z/N\Z$ and $f\colon \Z/N\Z\to\C$ we define $\Delta_h f\colon\Z/N\Z\to\Z$ as $\Delta_{h}f(n)= f(n+h)\overline{f(n)}$ for all $n\in\Z/N\Z$. 
Given $s\in\N$, the {\em Gowers uniformity norm} $\|.\|_{U^s_{\Z/N\Z}}$ on $\Z/N\Z$ is defined as
$$
\|f\|_{U^s_{\Z/N\Z}}:=\left(\frac{1}{N^{s+1}}\sum_{n,h_1,\ldots,h_s\in\Z/N\Z} \Delta_{h_1}\cdots\Delta_{h_s}f(n)\right)^{1/2^s}.
$$
To define the {\em Gowers uniformity seminorm} $\|.\|_{U^s_{[N]}}$ for a function $f\colon\N\to\C$, set $\tilde{N}:=2^s N$, define a function $f_N\colon \Z/\tilde{N}\Z\to\C$ as $f_N(n)=f(n)$ for $n\in[N]$ and $f_N(n)=0$ for $n\in[\tilde{N}]\setminus[N]$ (where we identify $\Z/\tilde{N}\Z$ with the interval $[\tilde{N}]$), let $1_{[N]}$ be the indicator function of the interval $[N]$, and define\footnote{We remark that there are different ways of introducing the Gowers uniformity seminorms $\|.\|_{U^s_{[N]}}$, but they all lead to equivalent notions of uniformity.
Also, for $s\geq 2$ and when viewed on the space of all functions $f\colon[N]\to\C$, the seminorm $\|.\|_{U^s_{[N]}}$ is in fact a norm. It will, however, be more convenient for us to view $\|.\|_{U^s_{[N]}}$ as a seminorm on the space of all bounded functions $f\colon\N\to\C$.
For a comprehensive discussion of this topic see subsections A.1 and A.2 of Appendix A
in \cite{FH17} or see Appendix B in \cite{GT10}.}
$$
\|f\|_{U^s_{[N]}}:=
\frac{\|f_N\|_{U^s_{\Z/\tilde{N}\Z}}}{\|1_{[N]}\|_{U^s_{\Z/\tilde{N}\Z}}}.
$$
A bounded function $f\colon\N\rightarrow \C$ is called
\define{$U^s$-uniform} if $\| f \|_{U^s_{[N]}}$ converges to zero as $N\to\infty$.
A function $f$ is called \define{uniform} if it is $U^s$-uniform for every $s\geq 1$.
\end{Definition}

It follows from \cite{GT12-2} and \cite{GTZ12} that the \define{M{\"o}bius function} $\mob$ and the \define{Liouville function} $\lio$ are uniform multiplicative functions.

We can state now the dichotomy theorem for $\mconv$.

\begin{Theorem}[Dichotomy theorem for $\mconv$]
\label{thm:dichotomy}
Let $f\in\mconv$. Then either
\begin{enumerate}
[label=(\roman*), ref=(\roman*), leftmargin=*]
\item\label{dich-1}
$f$ is \define{Besicovitch rationally almost periodic},
\end{enumerate}
or
\begin{enumerate}
[label=(\roman{enumi}), ref=(\roman{enumi}), leftmargin=*]
\setcounter{enumi}{1}
\item\label{dich-2}
$f$ is \define{uniform}.
\end{enumerate}
\end{Theorem}

Given a multiplicative function $f\colon\N\to\C$ and a point $z\in\C$ let $E(f,z)$ denote the set of solutions to the equation $f(n)=z$, i.e.,
\begin{equation*}
E(f,z):=\{n\in\N: f(n)=z\}.
\end{equation*}
We will refer to $E(f,z)$ as a \define{level set} of $f$ and we use $\Dzero$ to denote the collection of all sets of the form $E(f,z)$, where $f$ ranges over all multiplicative functions and $z$ ranges over all complex numbers.

\begin{Example}\label{example:mf}
Examples of sets belonging to $\Dzero$ include many classical sets of number-theoretical origin, such as: the \define{squarefree numbers} $Q:=\{n\in \N: p^2\nmid n~\text{for all primes $p$}\}$, the \define{multiplicatively even numbers} $\ME:=\{n\in\N:\Omega(n)~\text{is even}\}$  and the \define{multiplicatively odd numbers} $\MO:=\{n\in\N:\Omega(n)~\text{is odd}\}$, where $\Omega(n)$ denotes the number of prime factors of $n$ counted with multiplicities. In Subsection \ref{subsec:multiplicative-functions} we provide more examples of sets in $\Dzero$ (see\ \cref{Ex:multiplicative-functions} below). 
\end{Example}

Our main result is a structure theorem for sets belonging to $\Dzero$ in the spirit of \cref{thm:dichotomy}.
To formulate this theorem we first need to introduce set-theoretic analogues of uniform functions and of Besicovitch rationally almost periodic functions. 

\begin{Definition}[Uniform sets and relative uniformity]
\label{def:uniform-sets-relative-uniformity}
Let us call a set $A\subset\N$ \define{uniform} if $d(A):=\lim_{N\to\infty}\frac{|A\cap [N]|}{N}$ exists and the function $\1_A-d(A)$ is uniform in the sense of \cref{def:GN}.

Given $E,R\subset\N$ for which the densities $d(E)$ and $d(R)$ exist, we say that \define{$E$ is uniform relative to $R$} if $E\subset R$ and the function $d(R)\1_{E}-d(E)\1_R$ is uniform. Note that a set $A$ is uniform if and only if it is uniform relative to $\N$.
A more detailed discussion on the notion of relative uniformity can be found in Subsection \ref{subsec:conditional-uniformity}.
\end{Definition}

\begin{Example}
\label{eg:conditional-uniformity-0}
The multiplicatively even numbers $\ME$ and the multiplicatively odd numbers $\MO$ are examples of sets that are uniform (i.e.\ uniform relative to $\N$).
As an example of a set $E$ that is uniform relative to a set $R\subsetneq \N$, one can
take $E:=E(\mob,1)=\{n\in\N:\mob(n)=1\}$ and $R$ to be the set $Q$ of squarefree integers. Indeed, the function
$d(R)\1_E-d(E) \1_{R}$ is a scalar multiple of the M{\"o}bius function $\mob$, which is a uniform function.

\end{Example}

The structure theorem for $\Dzero$, which we will presently formulate, is motivated by \cref{eg:conditional-uniformity-0} and asserts that any set $E\in\Dzero$ is uniform relative to a ``structured'' superset $R$.
In this context, ``structured'' sets are elements of the family $\Drat$, which is introduced in the following definition.

\begin{Definition}
\label{def:rational-D_rat}
\label{def:c0rat}
\label{definition:rational-set}
\
\begin{enumerate}
[label=(\roman*), ref=(\roman*), leftmargin=*]
\item\label{def:itm:rational}
A set $A\subset\N$ is called \define{rational} if for every $\epsilon >0$ there exists a set $B$, that is a union of finitely many arithmetic progression, such that $\overline{d}(A\triangle B)<\epsilon$ (see {\cite[Definition 2.1]{BR02}} and \cite{BKLR16arXiv}).
Equivalently, a set $A$ is rational if and only
if its indicator function is Besicovitch rationally almost periodic.
\item\label{def:itm:D_rat}
Let $\Drat$ be the collection of all level sets $E(f,z)$, where $f$ is a Besicovitch rationally almost periodic multiplicative function and $z$ is an arbitrary complex number. We show in Subsection \ref{sec:D_rat} that any set in $\Drat$ is, in particular, a rational set.
\end{enumerate}
\end{Definition}

Before we state our main result we remark that the density of any rational set exists, which follows quickly from the definition of rationality, and the density of any set in $\Dzero$ also exists, which is a result established in \cite{Ruzsa77} (cf.\ \cref{cor:multiplicative-fibers-have-density} below).

\begin{Maintheorem}[Structure theorem for $\Dzero$]
\label{thm:structure-theorem}
For any set $E\in\Dzero$ with positive density there exists $R\in\Drat$ such that $E$ is uniform relative to $R$. If $d(E)\neq 1$ then $R\in\Drat$ with this property is unique.  
\end{Maintheorem}

\cref{thm:structure-theorem} allows us to study multiple ergodic averages along level sets of multiplicative functions, such as
\begin{equation}
\label{eqn:erg-ave-along-arth-sets}
\frac{1}{N}\sum_{n=1}^N \1_E(n)~T^{-p_1(n)}f_1 \cdots T^{-p_\ell(n)}f_\ell,
\end{equation}
where $T$ is an invertible measure preserving transformation on a probability space
$\xbm$, $f_1,\ldots,f_\ell \in L^\infty\xbm$, $p_1,\ldots,p_\ell$ are
polynomials with integer coefficients and $E$ belongs to $\Dzero$.

\begin{Definition}
\label{def:averaging-set-polinomial-multiple-recurrence}
Let $E=\{n_1<n_2<\ldots\}$ be a subset of $\N$.
We say that $E$ is an \define{averaging set of recurrence} if
for all invertible measure preserving systems $\xbmt$ and all $A\in\mathcal{B}$ with $\mu(A)>0$,
\begin{equation*}\label{limsuperior-single}
\lim_{N\to\infty}\frac{1}{N}\sum_{j=1}^N
\mu\big(A\cap T^{-n_j}A\big)>0.
\end{equation*}
We say that $E$ is an
\define{averaging set of polynomial multiple recurrence} if
for all invertible measure preserving systems $\xbmt$, all $A\in\mathcal{B}$ with
$\mu(A)>0$, all $\ell\geq 1$ and all polynomials
$p_i\in\Z[x]$, $i=1,\ldots,\ell$, with $p_i(0)=0$, we have
\begin{equation}\label{limsuperior-2}
\lim_{N\to\infty}\frac{1}{N}\sum_{j=1}^N
\mu\big(A\cap T^{-p_1(n_j)}A\cap\ldots\cap T^{-p_\ell(n_j)}A\big)>0.
\end{equation}
\end{Definition}



If $E$ is an averaging set of recurrence whose density $d(E)$ exists and is positive then it follows -- by considering cyclic rotations on finitely many points --
that the density of $E\cap u\N$ also exists and is positive for any positive integer $u$.
This divisibility property is a rather trivial but necessary condition for a positive density set to be ``good'' for averaging recurrence. This leads to the following definition.

\begin{Definition}
\label{def_divisibility-property}
Let $E\subset\N$. We say that $E$ is \emph{divisible}
if $d(E\cap u\N)$ exists and is positive for all $u\in\N$.
\end{Definition}

Expressions similar to \eqref{eqn:erg-ave-along-arth-sets} have also been studied by Frantzikinakis and Host in \cite{FH17-2}, where among other things they obtained the following result.

\begin{Theorem}[{\cite[Theorem 1.2, part (i)]{FH17-2}}]
\label{thm:FH-theorem-1.3-recurrence}
Let $k\in\N$ and let $f$ be a multiplicative function taking values in the set of $k$-th roots of unity. Suppose $f^j$ is aperiodic for all $j\in\{1,\ldots,k-1\}$ (i.e. $\lim_{N\to\infty}\frac1N\sum_{n=1}^N \allowbreak f^j(qn+r)=0$ for all $q\in\N$ and $r\in\N\cup\{0\}$) and $z$ is a point in the image of $f$. Then $E(f,z)=\{n\in\N:f(n)=z\}$ is an averaging set of polynomial multiple recurrence.   
\end{Theorem}

It is straightforward to verify that if $f$ and $z$ are as in \cref{thm:FH-theorem-1.3-recurrence} then the level set $E(f,z)$ is divisible\footnote{Indeed, if $f$ takes values in the $k$-th roots of unity $\{1,\zeta,\zeta^2,\ldots ,\zeta^{k-1}\}$ then $\1_{E(f,\zeta^i)}(n)=\frac{1}{k}\sum_{j=0}^{k-1}f^j(n)\zeta^{-ij}$ and therefore, using the fact that $f^j$ is aperiodic for $j\in\{1,\ldots,k-1\}$, we get that $d(E(f,\zeta^i)\cap u\N)=\lim_{N\to\infty}\frac{1}{k}\sum_{j=0}^{k-1}\frac{1}{N}\sum_{n=1}^N f^j(n)\zeta^{-ij}\1_{u\N}(n)=\frac{1}{uk}$.}.
In light of this fact, the next result, which is obtained by combining \cref{thm:structure-theorem} with the results obtained by the authors in \cite{BKLR16arXiv}, can be viewed as a generalization of \cref{thm:FH-theorem-1.3-recurrence}.

\begin{Maincorollary}
\label{thm:recurrence-divisibility}
Let $E\in\Dzero$ have positive density and let $r\in\N\cup\{0\}$. Then the following are equivalent:
\begin{itemize}
\item
$E-r$ is divisible;
\item
$E-r$ is an averaging set of recurrence;
\item
$E-r$  is an averaging set of polynomial multiple recurrence.
\end{itemize}
\end{Maincorollary}

In view of \cref{thm:recurrence-divisibility}, it is of interest to determine for which integers $r$ the set $E-r$ is divisible.

\begin{Example}\label{eg:motivation-for-thm:recurrence-enhanced}
Consider the sets $E:=E(\mob,1)=\{n\in\N:\mob(n)=1\}$ and $R:=Q$ from \cref{eg:conditional-uniformity-0}.
One can show that $E-r$ is divisible if and only if $r\in Q$. An analogous fact is true for the level set $E(\mob,-1)=\{n\in\N:\mob(n)=-1\}$.
\end{Example} 

The next proposition asserts that a phenomenon very similar to the one showcased in \cref{eg:motivation-for-thm:recurrence-enhanced} holds for any $E\in\Dzero$ of positive density.

\begin{Proposition}
\label{thm:recurrence-enhanced}
Suppose $E\in\Dzero$ has positive density. Let $R\in\Drat$ be as guaranteed by \cref{thm:structure-theorem}. Then for all $r\in R$ the set $E-r$ is divisible.
\end{Proposition}

\begin{Remark}
Note that in \cref{eg:motivation-for-thm:recurrence-enhanced} one has that if $r\notin Q$ then $E-r$ is not divisible. It is therefore natural to ask whether for any $E\in\Dzero$ and $r\notin R$ (where $R\in\Drat$ is as guaranteed by \cref{thm:structure-theorem}) the shift $E-r$ is not divisible. The answer, however, is negative (see \cref{eg:counterexample-selfshifts} below).
\end{Remark}

In \cite{BR02}, it was proven by the first author and Ruzsa that every self-shift of the set of squarefree numbers $Q$ (i.e. any set of the form $Q - r$ for $r \in  Q$) is good for polynomial multiple recurrence. Combining \cref{thm:recurrence-divisibility} and \cref{thm:recurrence-enhanced} yields a result of similar nature for all sets of positive density belonging to $\Dzero$.

\begin{Maincorollary}
\label{thm:recurrence}
Suppose $E\in\Dzero$ has positive density. Then every self-shift of $E$ is an averaging set of polynomial multiple recurrence. 
\end{Maincorollary}

\cref{thm:recurrence}, in turn, implies -- via Furstenberg's correspondence principle (see \cite[Theorem 1.1]{Bergelson87}) -- the following combinatorial result (cf. \cite{BL96} and \cite[Proposition 4.2]{BKLR16arXiv}).

\begin{Maincorollary}
\label{cor:comb}
Let $E$ be a set that belongs to $\Dzero$ and suppose $E$ has positive density.
Then for any set $D\subset \N$ with positive upper density, any polynomials $p_i\in\Q[t]$, $i=1,\ldots,\ell$, which satisfy
$p_i(\Z)\subset\Z$ and $p_i(0)=0$ for all $i\in\{1,\ldots,\ell\}$, and any $r\in E$
there exists $\beta>0$ such that the set
$$
\left\{n\in E-r:\overline{d}\Big(
D\cap (D-p_1(n))\cap \ldots\cap(D-p_\ell(n))
\Big)>\beta \right\}
$$
has positive lower density.
\end{Maincorollary}

\paragraph{Structure of the paper:}
\

In \cref{sec:prelim} we review basic results and facts regarding multiplicative functions, almost periodic functions and the Gowers uniformity seminorms, which are needed in the subsequent sections.

In \cref{sec:dt} we discuss the dichotomy between structure and randomness for multiplicative functions belonging to the class $\mconv$ and provide a proof of \cref{thm:dichotomy}.

In \cref{sec:structure-theorem} we discuss in detail the notion of relative uniformity and prove \cref{thm:structure-theorem}.

Finally, \cref{sec:ET} contains applications of our main results to the theory of multiple recurrence. In particular, we provide proofs for \cref{thm:recurrence-divisibility} and \cref{thm:recurrence-enhanced}.


\paragraph{Acknowledgements:}
We thank the anonymous referee for valuable remarks and suggestions.
We also thank Viktor Losert for providing several helpful comments and additional references regarding \cref{thm:besicovitch-approximation-theorem} in Subsection \ref{subsec:AP}.

\section{Preliminaries}\label{sec:prelim}

In this section we present some basic results and ideas, which will be used in the subsequent sections. In certain instances the classical results are presented in a slightly modified form and in those cases proofs are provided.

\subsection{Multiplicative functions}
\label{subsec:multiplicative-functions}

Recall that $\cm$ denotes the set of all multiplicative functions $f:\N\to\C$ with $|f(n)|\leq1$ for all $n\in\N$.
The set $\cm$ can be endowed with a ``distance'' function $\mathbb{D}:\cm\times\cm\to[0,\infty]$, which serves as a useful tool for cataloging the class of multiplicative functions bounded in modulus by $1$. 
Let $\P$ denote the set of prime numbers. For $f,g\in\cm$ define
\[
\mathbb{D}(f,g):= \sqrt{\sum_{p\in\P} \frac{1}{p} \Big(1- \Re(f(p)\overline{g(p)})\Big)}.
\]

\begin{Remark}\label{rem:basic-properties-of-D}
Let us list some important properties of $\mathbb{D}$. For more details and proofs the reader is referred to the book of Granville and Soundararajan \cite{GSdraft}.
\begin{enumerate}
[label=\text{(\arabic*)}, ref=\text{(\arabic*)}, leftmargin=*]
\item
$\mathbb{D}(f,g)=\mathbb{D}(g,f)=\mathbb{D}(\ov{f},\ov{g})$;
\item\label{item:rem:basic-properties-of-D-2}
$\mathbb{D}$ satisfies the triangle inequality, $\mathbb{D}(f,g)\leq \mathbb{D}(f,h)+\mathbb{D}(h,g)$;
\item\label{item:rem:basic-properties-of-D-3}
$m \mathbb{D}(f,g)\geq \mathbb{D}(f^m,g^m)$ for all $m\in\N$;
\item\label{item:rem:basic-properties-of-D-5}
$\mathbb{D}(f,g)<\infty$ implies $\mathbb{D}(|f|,|g|)<\infty$.
\end{enumerate}
\end{Remark}

When $\mathbb{D}(f,g)<\infty$ then, borrowing the terminology from \cite{GSdraft}, we say that $f$ \emph{pretends} to be $g$. In this case, many properties of $f$ are shared by $g$ and vice versa. For instance, we will see later that if $f$ pretends to be $g$, then $f$ is aperiodic if and only if $g$ is aperiodic (see\ \cref{def:aperiodic} and \cref{rem:uniform-aperiodic-D} below).

In Subsections \ref{sec:prelim-mean-value-thms}, \ref{subsec:AP} and \ref{sec:prelim-2} below we will see that one can often determine whether a multiplicative function $f\in\cm$:
\begin{itemize}
\item[--]
has a mean value,
\item[--]
is Besicovitch almost periodic,
\item[--]
is aperiodic, or
\item[--]
is uniform
\end{itemize}
by measuring the $\mathbb{D}$-distance between $f$ and \define{Archimedean characters} and \define{Dirichlet characters}.  
An \define{Archimedean character} is a function of the form $n\mapsto n^{it}=e^{it\log n}$ with $t\in\R$. Any Archimedean character is a completely multiplicative element of $\cm$.
An arithmetic function $\chi$ is called a \define{Dirichlet character} if there exists a number $d\in\N$, called a \define{modulus of $\chi$}, such that:
\begin{enumerate}
[label=\text{(\arabic*)\quad}, ref=\text{(\arabic*)}, leftmargin=*]
	\item 	$\chi(n+d)=\chi(n)$ for all $n\in\N$;
	\item 	$\chi(n)=0$ whenever $\gcd(d,n)>1$, and $\chi(n)$ is a $\tot(d)$-th root of unity whenever $\gcd(d,n)=1$, where $\tot$ denotes Euler's totient function;
	\item	$\chi(nm)=\chi(n)\chi(m)$ for all $n,m\in\N$.
\end{enumerate}
Any Dirichlet character is periodic and completely multiplicative.\footnote{The converse of this statement is also true: Any periodic and completely multiplicative function is a Dirichlet character.}
We also remark that $\chi\colon \N\to\C$ is a Dirichlet character of modulus $k$ if and only if there exists a group character $\widetilde{\chi}$ of the multiplicative group $(\Z/k\Z)^*$ such that $\chi(n)=\widetilde{\chi}(n~{\rm mod}~k)$ for all $n\in\N$.
The Dirichlet character determined by the trivial (constant equal to $1$) character of $(\Z/k\Z)^*$ is called the {\em principal character of modulus $k$}. It is denoted by $\chi_1$. Note that if $d|k$ and $\chi$ is a Dirichlet character of modulus $d$ then
\begin{equation}\label{eq:induced-character}
\chi':=\chi\cdot \chi_1
\end{equation}
is a Dirichlet character of modulus $k$.
Throughout this paper we reserve the letter $\chi$
to denote Dirichlet characters.

\begin{Lemma}[{cf.\ \cite[Lemma 4.6]{GSdraft} and \cite[Remark after Lemma 2.2]{FKLdraft}}]
\label{item:rem:basic-properties-of-D-4}
\label{DirArch}
For every $t\neq0$ and every Dirichlet character $\chi$ we have $\mathbb{D}(\chi,n^{it})=\infty$. In particular, for $t\neq 0$ one has $\mathbb{D}(1,n^{it})=\infty$.
\end{Lemma}

We end this subsection with a list containing examples of multiplicative functions belonging to $\cm$ and examples of sets in $\Dzero$ that can be obtained from functions in $\cm$. 

\begin{Example}
\label{Ex:multiplicative-functions}
\
\begin{enumerate}
[label=(\arabic*), ref=(\arabic*), leftmargin=*]
\item\label{item:eg:multiplicative-functions-1}
The \define{Liouville function}
$\lio$ is defined as $\lio(n):=(-1)^{\Omega(n)}$ and is completely multiplicative (for the definition of $\Omega(n)$ see \cref{example:mf}). The non-trivial level sets of $\lio$ are exactly the multiplicatively even and odd numbers $\ME$ and $\MO$ defined in \cref{example:mf}.
\item\label{item:eg:multiplicative-functions-1.5}
The \define{M{\"o}bius function} $\mob$ is defined as $\mob(n):=\lio(n)\1_Q(n)$. Note that $\mob$ is multiplicative but not completely multiplicative.
\item\label{item:eg:multiplicative-functions-5}
Throughout this paper we identify the torus $\T:=\R/\Z$ with the unit interval $[0,1)\bmod 1$ or, when convenient, with the unit circle in the complex plane. Also, we introduce the notation $e(x):=e^{2\pi i x}$ for all $x\in\R$.
Given $\xi\in\T$, define the multiplicative functions $\lio_{\xi}$, $\mob_{\xi}$ and $\boldsymbol{\kappa}_\xi$ as
$$
\lio_{\xi}(n):=e( \xi\Omega(n)),\qquad
\mob_{\xi}(n):=\lio_{\xi}(n)\1_{Q}(n),
$$
and
$$
\boldsymbol{\kappa}_\xi(n):=e( \xi\omega(n)),
$$
where $\omega(n)$ denotes the number of distinct prime divisors of $n$ (counted without multiplicities). It is clear that $\boldsymbol{\kappa}_\xi,\lio_{\xi},\mob_{\xi}\in\cm$.
Observe that $\lio_{\frac12}=\lio$ and $\mob_{\frac12}=\mob$.

The following examples of sets belong to $\Dzero$ because they can be viewed as level sets of the functions $\boldsymbol{\kappa}_\xi$, $\lio_{\xi}$, and $\mob_{\xi}$, respectively, where $\xi$ is any primitive $b$-th root of unity:
\begin{eqnarray*}
S_{\omega,b,r}&:=&\{n\in\N: \omega(n)\equiv r\bmod b\},
\\
S_{\Omega,b,r}&:=&\{n\in\N: \Omega(n)\equiv r\bmod b\},
\\
U_{b,r}&:=&\{n\in\N:n~\text{is squarefree and}~\Omega(n)\equiv r\bmod b\}.
\end{eqnarray*}
Note that $\ME=S_{\Omega,2,0}$ and $\MO=S_{\Omega,2,1}$.
\item\label{item:eg:multiplicative-functions-11}
If $f:\N\to\N$ is multiplicative and $b$ is either $2$, $4$, $p$ or $2p$, where $p$ stands for an odd prime number, then for any $r\in\{0,1,\ldots,b-1\}$ with $\gcd(b,r)=1$ the set
$$
V_{f,b,r}:=\{n\in\N: f(n)\equiv r\bmod b\}
$$ 
is an element of $\Dzero$. This is because for any such $b$ the multiplicative group of integers mod $b$ is cyclic and hence there exists a Dirichlet character $\chi$ which spans $(\Z/b\Z)^*$. Therefore $V_{b,r}$ can be realized as a level set of the multiplicative function $\chi\circ f$, which belongs to $\cm$.
In particular, the set
\begin{eqnarray*}
S_{\boldsymbol{\tau},b,r}&:=&\{n\in\N: \boldsymbol{\tau}(n)\equiv r\bmod b\}
\end{eqnarray*}
belongs to $\Dzero$, where $\boldsymbol{\tau}(n):=\sum_{d\mid n} 1$ is the \define{number of divisors function}.
\end{enumerate}
\end{Example}

\subsection{Mean value theorems of Wirsing and Hal{\'a}sz}\label{sec:prelim-mean-value-thms}

We say a function $f\in\cm$ has a \define{mean value}, and denote it by $M(f)$, if the limit
\begin{equation}
\label{eqn:mean-value}
M(f):=\lim_{N\to\infty}\frac1N\sum_{n=1}^N f(n)
\end{equation}
exists.
In general, the mean value of a multiplicative function $f\in\cm$ does not exist; for example, if $t\neq 0$ then the mean of $n^{it}$ does not exist, cf.\ \cite[Section 4.3]{GSdraft}.

Two classical theorems in multiplicative number theory are Wirsing's celebrated mean value theorem regarding real-valued multiplicative functions bounded in modulus by $1$, and Hal{\'a}sz's generalization of Wirsing's theorem to all functions in $\cm$.

\begin{Theorem}[Wirsing; see \cite{Wirsing61} and {\cite[Theorem 6.4]{Elliott79}}]
\label{thm:wirsing}
For any real-valued $g\in\cm$ the mean value $M(g)$ exists.
\end{Theorem}

\begin{Theorem}[Hal{\'a}sz; see {\cite[Theorem 6.3]{Elliott79}}]
\label{thm:FH-thm.2.9}
Let $g\in\cm$. Then the mean value $M(g)$
exists if and only if one of the following mutually exclusive conditions is satisfied:
\begin{enumerate}
[label=\text{(\roman{enumi})}, ref=\text{(\roman{enumi})}, leftmargin=*]
\item\label{item:E79-thm6.3-i}
there is at least one positive integer $k$ so that $g(2^k)\neq-1$ and, additionally, the series $\sum_{p\in\P}\frac1p(1-g(p))$ converges;
\item
\label{item:E79-thm6.3-iii}
there is a real number $t$ such that $\mathbb{D}(g,n^{it})<\infty$ and, moreover, for each positive integer $k$ we have $g(2^k)=-2^{itk}$;
\item
\label{item:E79-thm6.3-iv}
$\mathbb{D}(g,n^{it})=\infty$ for each $t\in\R$.
\end{enumerate}
When condition \ref{item:E79-thm6.3-i} is satisfied then $M(g)$ is non-zero and can be computed explicitly using the formula
\begin{equation}
\label{eq:mean-value-0}
M(g)=\prod_{p\in\P}\left(1-\frac1p\right)\left(1+\sum_{m=1}^\infty p^{-m}g(p^m)\right).
\end{equation}
In the case when $g$ satisfies either \ref{item:E79-thm6.3-iii} or \ref{item:E79-thm6.3-iv} then the mean value $M(g)$ equals zero.
\end{Theorem}

\subsection{Besicovitch almost periodic functions}
\label{subsec:AP}

The Besicovitch seminorm $\|\cdot\|_B$ and Besicovitch almost periodic and Besicovitch rationally almost periodic functions were introduced in \cref{sec:intro} (see equation \eqref{eq:def-besic-seminorm} and \cref{def:d3}).

Any periodic function is clearly Besicovitch rationally almost periodic. In particular, any Dirichlet character $\chi$ is Besicovitch rationally almost periodic. There are, however, many other natural examples of multiplicative functions that are Besicovitch rationally almost periodic. For instance, $\mob^2$ and $\tfrac{\tot(n)}{n}$ are such. More generally, it will be shown at the end of this subsection (see \cref{ren:0-1-valued-RAP} below) that any bounded multiplicative function with values in $[0,\infty)$ is Besicovitch rationally almost periodic.

For any Besicovitch almost periodic function $f\colon \N\to\C$ and any $\theta\in[0,1)$ the limit
$$
\hat f(\theta):=\lim_{N\to\infty}\frac{1}{N}\sum_{n=1}^Nf(n)e(-n\theta)
$$
exists; moreover, $\hat f(\theta)$ differs from $0$ for at most countably many values of $\theta$ (cf.\ \cite[pp. 104 -- 105]{Besicovitch55}). The set $\sigma(f):=\{\theta\in[0,1):\hat{f}(\theta)\neq 0\}$ is called the \define{spectrum} of $f$.
See \cite{BL85,Besicovitch55} for more information on the Fourier analysis of almost periodic functions.

We say that a Besicovitch almost periodic function $f\colon \N\to\C$ has \define{rational spectrum} if $\sigma(f)$ is a subset of $\Q\cap[0,1)$.
Note that if $f$ is periodic then its spectrum is rational. Note also that each Besicovitch almost periodic function $f$ has a mean given by $\hat{f}(0)$. It easily follows that there are $f\in\cm$ that are not Besicovitch almost periodic (indeed, take any $f\in\cm$ which has no mean). However, it follows from the next theorem, which is due to Daboussi, that whenever $f\in\cm$ is Besicovitch almost periodic then its spectrum has to be rational (this fact is used later, cf.\ \cref{cor:DD82-thm1-and-thm6} part \ref{item:DD82-thm1-and-thm6-i} and \ref{item:DD82-thm1-and-thm6-ii}).

\begin{Theorem}[cf.\ {\cite[Theorem 1]{DD82}}]
\label{thm:DD82-thm.1.}
Let $f\in\cm$. Then for all irrational $\theta$,
$$
\lim_{N\to\infty}\frac{1}{N}\sum_{n=1}^N f(n)e(\theta n)=0.
$$
\end{Theorem}

In \cref{cor:RAP-have-rational-spectrum} below we show that a Besicovitch almost periodic function is Besicovitch rationally almost periodic if and only if it has rational spectrum. We will derive this as a corollary from the following theorem.

\begin{Theorem}[cf.\ {\cite[Theorem II.8.2\ensuremath{^\circ}(page 105)]{Besicovitch55}} and {\cite[Lemma 3.11]{BL85}}]\label{thm:besicovitch-approximation-theorem}
Let $f\colon \N\to\C$ be a Besicovitch almost periodic function with spectrum $\sigma(f)$.
Then for every $\epsilon>0$ there exists a trigonometric polynomial $P(n)=\sum_{i=1}^k c_ie(\theta_i n)$ with
$c_1,\ldots,c_k\in\C$ and $\theta_1,\ldots,\theta_k\in\sigma(f)$ such that $\|f-P\|_B\leq \epsilon$.
This includes the case $\sigma(f)=\emptyset$, where one can take $P\equiv 0$ for all $\epsilon>0$ (i.e., $f$ has empty spectrum if and only if $\|f\|_B=0$).
\end{Theorem}

\begin{proof}The following proof was kindly provided to the authors by Viktor Losert.

Given $n_0,n_1,\ldots,n_s\in\N$ and $\beta_1,\ldots,\beta_s\in\R$ such that the set $\{2\pi,\beta_1,\beta_2,\ldots,\beta_s\}$ is linearly independent over $\Q$, the \define{discrete Bochner-Fej{\'e}r kernel} with parameter $B=\left(\begin{smallmatrix}n_0 & n_1 & \ldots & n_s \\ 2\pi & \beta_1 & \ldots & \beta_s\end{smallmatrix}\right)$ is defined as
$$
K_B(k):=\sum \left(1-\frac{|\nu_1|}{n_1}\right)\cdots \left(1-\frac{|\nu_s|}{n_s}\right)e^{-i\left(\frac{\nu_0}{n_0}2\pi+\nu_1\beta_1+\ldots+\nu_s\beta_s\right)k},
$$
where the sum ranges over $\nu_0=1,\ldots,n_0$, $|\nu_1|<n_1, \ldots, |\nu_s|<n_s$. The corresponding \define{discrete Bochner-Fej{\'e}r polynomial} is
$$
\sigma_{B}^f(n):=\lim_{N\to\infty}\frac{1}{N}\sum_{k=1}^{N} f(n+k)K_B(k).
$$

It is shown in {\cite[Lemma 3.11]{BL85}} (also cf.\ {\cite[Theorem II.8.2\ensuremath{^\circ}(page 105)]{Besicovitch55}}) that there exists a sequence of Bochner-Fej{\'e}r polynomials $\sigma_{B_m}^f$, $m\in\N$, such that $\|f- \sigma_{B_m}^f\|_B\to 0$ as $m\to\infty$.
It is not hard to see that
$$
\widehat{\sigma_{B_m}^f}(\theta)~=~\hat{f}(\theta)\cdot \widehat{K_{B_m}}(1-\theta)~=~\hat{f}(\theta)\cdot \widehat{K_{B_m}}(\theta)
$$
and hence $\sigma(\sigma_{B_m}^f)\subset\sigma(f)$. This finishes the proof.
%
\end{proof}

From \cref{thm:besicovitch-approximation-theorem} we obtain the following corollary.

\begin{Corollary}\label{cor:RAP-have-rational-spectrum}
Let $f\colon \N\to\C$ be Besicovitch almost periodic. Then $f$ is Besicovitch rationally almost periodic if and only if $f$ has rational spectrum.
\end{Corollary}

\begin{proof}
First assume $f$ has rational spectrum. By \cref{thm:besicovitch-approximation-theorem}, $f$ can be approximated in the seminorm $\|\cdot\|_B$ by trigonometric polynomials of the from $P(n)=\sum_{i=1}^k c_ie(\theta_i n)$ with $c_1,\ldots,c_k\in\C$ and $\theta_1,\ldots,\theta_k\in\sigma(f)\subset\Q$. Since $\theta_1,\ldots,\theta_k$ are rational numbers, the functions $P(n)$ is periodic. In other words, $f$ satisfies the definition of Besicovitch rationally almost periodic functions.

Next, let $f$ be Besicovitch rationally almost periodic and let $\theta$ be an irrational number. We will show that $\hat{f}(\theta)=0$. Let $\epsilon>0$ be arbitrary and let $P\colon \N\to\C$ be a periodic function with $\|f-P\|_B\leq \epsilon$. Then
\begin{eqnarray*}
\big|\hat{f}(\theta)\big|
&=& \lim_{N\to\infty}\left|\frac{1}{N}\sum_{n=1}^Nf(n)e(-\theta n)\right|\\
&\leq & \lim_{N\to\infty}\left|\frac{1}{N}\sum_{n=1}^N P(n)e(-\theta n)\right| +\epsilon\\
&=&\epsilon.
\end{eqnarray*}
Since $\epsilon>0$ was chosen arbitrarily, we conclude that $\hat{f}(\theta)=0$. This shows that no irrational number $\theta$ belongs to $\sigma(f)$.
\end{proof}


The next lemma is a consequence of \cref{thm:FH-thm.2.9} and establishes a connection between the distance function $\mathbb{D}$, defined in Subsection \ref{subsec:multiplicative-functions},
and the Besicovitch seminorm $\|\cdot\|_B$.

\begin{Lemma}\label{lem:lH2-C}
Suppose $f\in\cm$. Then $\|f\|_B=0$ if and only if
$\mathbb{D}(|f|,1)=\infty$.
\end{Lemma}
\begin{proof}
First, observe that $\|f\|_B=0$ if and only if the mean value of the multiplicative function $|f|$ is zero, i.e., $M(|f|)=0$. In view of \cref{thm:FH-thm.2.9}, the mean value of $|f|$ is zero if and only if $|f|$ satisfies either condition \ref{item:E79-thm6.3-iii} or condition \ref{item:E79-thm6.3-iv} of the theorem.
Since
$$
1-|f(p)|\cos\left(t{\log(p)}\right)\geq \min\left(1,1-\cos\left(t{\log(p)}\right)\right),\qquad\forall p\in\P,
$$
and $\mathbb{D}(1,n^{it})=\infty$ for all $t\neq 0$ (cf.\ \cref{item:rem:basic-properties-of-D-4}), it follows that $\mathbb{D}(|f|,n^{it})=\infty$ for all $t\neq 0$.
Therefore $|f|$ cannot satisfy condition \ref{item:E79-thm6.3-iii} of \cref{thm:FH-thm.2.9}. Hence $M(|f|)=0$ if and only if $f$ satisfies condition \ref{item:E79-thm6.3-iv} of \cref{thm:FH-thm.2.9}. Finally, observe that $|f|$ satisfies condition \ref{item:E79-thm6.3-iv} if and only if $\mathbb{D}(|f|,1)=\infty$.
\end{proof}

In \cite{DD82, DD74} Daboussi and Delange give necessary and
sufficient conditions for a bounded multiplicative function to be Besicovitch almost periodic:

\begin{Theorem}[{\cite[Theorem 6]{DD82}}]
\label{thm:DD82-thm6}
A function $f\in\cm$ is Besicovitch almost periodic if and only if
either $\|f\|_B=0$ or there exists a Dirichlet character $\chi$ such that
$\sum_{p\in\P}\frac{1}{p}(1-f(p)\overline{\chi(p)})$ converges.\footnote{Actually, Daboussi and Delange prove their theorem for the larger class of multiplicative functions satisfying $\sum_{n\leq x} |f(n)|^2=\Oh(x)$.}
\end{Theorem}

From \cref{thm:DD82-thm6} we obtain the following corollary.

\begin{Corollary}
\label{cor:DD82-thm1-and-thm6}
Let $f\in\cm$. The following are equivalent:
\begin{enumerate}
[label=\text{(\roman*)}, ref=\text{(\roman*)}, leftmargin=*]
\item\label{item:DD82-thm1-and-thm6-i}
$f$ is Besicovitch almost periodic;
\item\label{item:DD82-thm1-and-thm6-ii}
$f$ is Besicovitch rationally almost periodic;
\item\label{item:DD82-thm1-and-thm6-iii}
either $\|f\|_B=0$ or there exists a Dirichlet character $\chi$ such that
$\sum_{p\in\P}\frac{1}{p}(1-f(p)\overline{\chi(p)})$ converges.
\end{enumerate}
\end{Corollary}

\begin{proof}
The equivalence of \ref{item:DD82-thm1-and-thm6-ii} and \ref{item:DD82-thm1-and-thm6-iii} is given by \cref{thm:DD82-thm6}. Also, the fact that \ref{item:DD82-thm1-and-thm6-ii} implies \ref{item:DD82-thm1-and-thm6-i} is obvious.
It thus remains to show that \ref{item:DD82-thm1-and-thm6-i} implies \ref{item:DD82-thm1-and-thm6-ii}. However, from \cref{thm:DD82-thm.1.} we deduce that any multiplicative function $f$ has rational spectrum, which, in view of \cref{cor:RAP-have-rational-spectrum}, implies that $f$ is Besicovitch rationally almost periodic.
\end{proof}

\begin{Remark}\label{ren:0-1-valued-RAP}
We claim that any bounded multiplicative function $f$ taking values in $[0,\infty)$ is Besicovitch rationally almost periodic.

Let us first prove the claim for the special case when $0\leq f(n)\leq 1$ for all $n\in \N$. If $\Vert f\Vert_B =0$ then $f$ is Besicovitch rationally almost periodic for trivial reasons; it thus suffices to verify the claim for $f$ with $\Vert f\Vert_B >0$. In view of \cref{lem:lH2-C} it follows from $\Vert f\Vert_B >0$ that $\mathbb{D}(f,1)<\infty$. Since $f$ only takes values in the interval $[0,1]$, the assertion $\mathbb{D}(f,1)<\infty$ is equivalent to the fact that the series $\sum_{p\in\P}\frac{1}{p}(1-f(p))$ converges. Therefore, using \ref{item:DD82-thm1-and-thm6-iii} $\Rightarrow$ \ref{item:DD82-thm1-and-thm6-ii} of \cref{cor:DD82-thm1-and-thm6}, we conclude that $f$ is Besicovitch rationally almost periodic.

Next, assume $f$ takes values in $[0,b)$ for some $b\geq 1$. Define two new multiplicative functions $g$ and $h$ via 
$$
g(p^k):=
\begin{cases}
f(p^k),&\text{if}~f(p^k)\leq 1
\\
1,&\text{if}~f(p^k) >1
\end{cases}
\qquad\text{and}\qquad
h(p^k):=
\begin{cases}
1,&\text{if}~f(p^k)\leq 1
\\
f(p^k),&\text{if}~f(p^k) >1.
\end{cases}
$$
Clearly, $f(n)=g(n)h(n)$for all $n\in\N$. Moreover, $g$ and $\frac{1}{h}$ are multiplicative functions taking values in $[0,1]$. It follows from the previous paragraph that both $g$ and $\frac{1}{h}$ are Besicovitch rationally almost periodic.
Since $\frac{1}{h}$ is Besicovitch rationally almost periodic, for every $\epsilon>0$ there exists a periodic function $P$ with $\|\frac{1}{h}-P\|_B\leq \epsilon$. Since $\frac{1}{b}\leq \frac{1}{h(n)}\leq 1$, we can assume without loss of generality that $\frac{1}{b}\leq P(n)\leq 1$. It is then straightforward to show that $\|h-\frac{1}{P}\|_B\leq b^2\epsilon$, which proves that $h$ is also Besicovitch rationally almost periodic. Finally, observe that $f$, as a product of two Besicovitch rationally almost periodic functions, is itself Besicovitch rationally almost periodic.
\end{Remark}

\subsection{Ruzsa's theorem and some of its corollaries}
\label{sec:ruzsa}

In this short section we formulate a theorem of Ruzsa that shows that the density of a level set of a multiplicative function always exists and which gives necessary and sufficient conditions for this density to be positive. We also derive additional corollaries from this theorem which will be used in the latter sections of this paper. 

Let $r\in\N$. A function $\vec{f}=(f_1,\ldots,f_r)\colon\N\to\C^r$ is called \define{multiplicative} if each of its coordinate components $f_i\colon\N\to\C$ is a multiplicative function. 
We say that a point $\vec{z}\in \C^r$ is a \define{concentration point} for a multiplicative function $\vec{f}\colon\N\to\C^r$ if the set $P:=\{p\in\P: \vec{f}(p)=\vec{z}\}$ satisfies $\sum_{p\in P}\frac{1}{p}=\infty$. In the following, we denote by $im(\vec{f})$ the image of $\vec{f}$ and, for $\vec{z}\in im(\vec{f})$, we use $E(\vec{f},\vec{z}):=\{n\in\N: \vec{f}(n)=\vec{z}\}$ to denote level sets of $\vec{f}$.

\begin{Definition}[cf.\ {\cite[Definition 3.8]{Ruzsa77}}]\label{def:concentration-pt-old}
\label{def:ST-concentrated-functions}
Assume that a multiplicative function $\vec{f}\colon\N\to(\C\setminus\{0\})^r$ possesses at least one concentration point $\vec{z}=(z_1,\ldots,z_r)$. The subgroup $\G$ of the multiplicative group $((\C\setminus\{0\})^r,\cdot)$ generated by all concentration points of $\vec{f}$ is called the \define{concentration group} of $\vec{f}$.
\end{Definition}


\begin{Theorem}[cf.\ {\cite[Theorem 3.10]{Ruzsa77}}]\label{thm:ST-Ruzsa-3.10}
Let $\vec{f}\colon \N\to (\C\setminus\{0\})^r$ be a multiplicative function.
\begin{enumerate}
[label=\text{(\arabic*)}, ref=\text{(\arabic*)}, leftmargin=*]
\item\label{itm_p1}
Assume that $\vec{f}$ satisfies the following three conditions:
\begin{enumerate}
[label=\text{(\alph*)}, ref=\text{(\alph*)}, leftmargin=*]
\item\label{itm_a}
$\vec{f}$ has at least one concentration point,
\item\label{itm_b}
the concentration group $\G$ of $\vec{f}$ is finite, and
\item\label{itm_c}
$
\sum_{\substack{p\in\P,\\ \vec{f}(p)\notin\G}}\frac{1}{p}<\infty.
$
\end{enumerate}
Then $d(E(\vec{f},\vec{z}))$ exists and is strictly positive for all $\vec{z}\in im(\vec{f})$. Moreover,
$$
\sum_{\vec{z}\in im(\vec{f})}d(E(\vec{f},\vec{z}))=1.
$$ 
\item
If $\vec{f}$ does not satisfy (at least) one of the conditions  \ref{itm_a}, \ref{itm_b} or \ref{itm_c} of part \ref{itm_p1}, then $d(E(\vec{f},\vec{z}))=0$ for all $\vec{z}\in im(\vec{f})$.
\end{enumerate}
\end{Theorem}

Although we formulated \cref{thm:ST-Ruzsa-3.10} for arbitrary $r\in \N$, we will mostly deal with the special case $r=1$; the only exception is the proof of \cref{lem:common-unit-value} below, where we also need the case $r=2$.

\begin{Corollary}[cf.\ {\cite[Corollary 1.6 and the subsequent remark]{Ruzsa77}}]
\label{cor:multiplicative-fibers-have-density}
For any multiplicative function $f\colon \N\to\C$ and any $z\in\C$ the density of $E(f,z)$ exists.
\end{Corollary}

\begin{Definition}[cf.\ {\cite[Definition 3.9]{Ruzsa77}}]
\label{def:concentrated}
A multiplicative function $\vec{f}\colon \N\to(\C\setminus\{0\})^r$ is called \define{concentrated} if it satisfies conditions \ref{itm_a}, \ref{itm_b} and \ref{itm_c} in part \ref{itm_p1} of \cref{thm:ST-Ruzsa-3.10}.
\end{Definition}

\begin{Corollary}
\label{cor:density->concentrated function}
Let $f\colon \N\to\C$ be a multiplicative function and $z\in\C\setminus\{0\}$. If $d(E(f,z))>0$ then there exists a concentrated multiplicative function $g\colon \N\to\C\setminus\{0\}$ such that
$$
E(f,z)=E(g,z).
$$
Moreover, the set $P:=\{p\in\P: f(p)\neq g(p)\}$ satisfies $\sum_{p\in P}\frac{1}{p}<\infty$.
\end{Corollary}

\begin{proof}
Define $J:=im(f)\setminus\{0\}$.
Since $J$ is a countable subset of $\C\setminus\{0\}$, there exists $y\in \C\setminus\{0\}$ such that $(y^n\cdot J)\cap J=\emptyset$ for all $n\in\N$. We define a new multiplicative function $g$ as
$$
g(p^k):=
\begin{cases}
f(p^k),&\text{if }f(p^k)\neq 0
\\
y,&\text{if }f(p^k)=0
\end{cases},
\qquad\forall k\in\N,~\forall p\in\P.
$$
It follows from $(y^n\cdot J)\cap J=\emptyset$ that $E(f,z')=E(g,z')$ for all $z'\in J$, so, in particular $E(f,z)=E(g,z)$.
Since $g(n)\neq 0$ for all $n\in\N$, we can apply \cref{thm:ST-Ruzsa-3.10} and deduce that $g$ must satisfy conditions \ref{itm_a}, \ref{itm_b} and \ref{itm_c} in part \ref{itm_p1} of \cref{thm:ST-Ruzsa-3.10}.

Note that $\|f\|_B>0$, because $z\neq 0$ and $d(E(f,z))>0$. Therefore, by \cref{lem:lH2-C}, we conclude that $P'=\{p\in\P:f(p)=0\}$ satisfies $\sum_{p\in P'}\frac{1}{p}<\infty$.
Finally, note that $f(p)\neq g(p)$ if and only if $p\in P'$, which completes the proof.
\end{proof}

\begin{Remark}\label{vector}
Consider $\vec{f}\colon\N\to(\C\setminus\{0\})^r$. Notice that, by Theorem~\ref{thm:ST-Ruzsa-3.10}, if $\vec{f}$ is concentrated then $d(E(\vec{f},\vec{1}))>0$, because $\vec{1}\in im(\vec{f})$. Moreover, $\vec{f}$ is concentrated if and only if each of its coordinates $f_i$ is concentrated.
\end{Remark}

\subsection{Uniform functions}
\label{sec:prelim-2}



It follows from the work of Green and Tao \cite{GT12-2} and Green, Tao and Ziegler \cite{GTZ12} that the classical M{\"o}bius function $\mob$ is a uniform function (see \cref{def:GN}). A more general result was
obtained by Frantzikinakis and Host in \cite{FH17}. In order to state their theorem, we need the following definition.

\begin{Definition}\label{def:aperiodic}
We call a multiplicative function $f$ \define{aperiodic}
if for all $b\in\N$ and all $r\in\{0,1,\ldots,b-1\}$ we have
$
\lim_{N\to\infty}\frac{1}{N}\sum_{n=1}^N f(bn+r)=0.
$
\end{Definition}

\begin{Theorem}[Theorem 2.4, \cite{FH17}]
\label{thm:tAU}
A multiplicative function $f\in \cm$ is uniform if and only if it is aperiodic.
\end{Theorem}

In \cite{Delange72, Delange83}, Delange gives a full characterization of all aperiodic functions in $\cm$:


\begin{Proposition} \label{prop:l10} Let $f\in\cm$. Then $f$ is aperiodic if and only if $\mathbb{D}(f,\chi\cdot n^{it})=\infty$ for each Dirichlet character $\chi$ and $t\in\R$.
\end{Proposition}


\begin{Remark}\label{rem:uniform-aperiodic-D}
It follows immediately from \cref{prop:l10} and from the triangle inequality for $\mathbb{D}$ (see \cref{rem:basic-properties-of-D}) that if $f,g\in\cm$
satisfy $\mathbb{D}(f,g)<\infty$ then $f$ is aperiodic if and only if $g$
is aperiodic. Using \cref{thm:tAU} we can replace ``aperiodic'' with ``uniform''.
Hence, we get that if $f,g\in\cm$ satisfy $\mathbb{D}(f,g)<\infty$ then
$f$ is uniform if and only if $g$ is uniform.
\end{Remark}

\begin{Proposition}\label{prop:interplay-uniform-almost-periodic}
Suppose $f\colon \N\rightarrow \C$ is bounded.
\begin{enumerate}
[label=\text{(\alph*)}, ref=\text{(\alph*)}, leftmargin=*]
\item\label{itm:prop:interplay-uniform-almost-periodic-a}
If $f_n$ is uniform and $f_n\to f$ in $\|\cdot\|_B$, then
$f$ is uniform.
\item\label{itm:prop:interplay-uniform-almost-periodic-b0}
If $f$ is uniform, $q\in \N$ and $r\in\{0,1,\ldots,q-1\}$ then $f\cdot\1_{q\N+r}$
is uniform.
\item\label{itm:prop:interplay-uniform-almost-periodic-b}
If $f$ is uniform and if $g$ is Besicovitch rationally almost periodic,
then 
$h:=f\cdot g$
is uniform.
\item\label{itm:prop:interplay-uniform-almost-periodic-c}
If $f$ is uniform and $t\in\N$ then $h(n):=f(tn)$ is uniform.
\end{enumerate}
\end{Proposition}

\begin{proof}
To prove part \ref{itm:prop:interplay-uniform-almost-periodic-a} it suffices to show that for all $f\colon \N\to\C$ bounded in modulus by $1$ we have
\begin{equation}
\label{eqn:U-B-ineq}
\|f_N \|_{U^s[N]}^{2^{s+1}}\leq  \frac{1}{N}\sum_{n=1}^N |f(n)|.
\end{equation}
We prove \eqref{eqn:U-B-ineq} by induction on $s$.
For $s=1$ the inequality in \eqref{eqn:U-B-ineq} follows immediately from the definition of 
the $U^1$-norm (see \cref{def:GN}). Thus, assume \eqref{eqn:U-B-ineq} has already been proven for $s\geq 1$. Then,
\begin{eqnarray*}
\|f\|_{U^{s+1}_{[N]}}^{2^{s+1}}
&=&\frac{1}{N}\sum_{h=1}^N \left\|f_N T^h \overline{f_N}\right\|_{U^s_{[N]}}^{2^s}\\
&\leq&\frac{1}{N}\sum_{h=1}^N \frac{1}{N}\sum_{n=1}^N\left |f_N(n) \overline{f_N(n+h)}\right|\\
&\leq& \frac{1}{N}\sum_{n=1}^N\left |f_N(n) \right|.
\end{eqnarray*}

Part \ref{itm:prop:interplay-uniform-almost-periodic-b0} follows directly from the inverse conjecture for the Gowers seminorms (see \cite{GTZ12} or \cite[Theorem 1.6.12 and Theorem 1.6.14]{Tao12}).

For the proof of part \ref{itm:prop:interplay-uniform-almost-periodic-b} observe that it follows from part \ref{itm:prop:interplay-uniform-almost-periodic-b0} and the triangle inequality for $\|\cdot\|_{U_{[N]}^s}$ that for any uniform $f$ and any $g$ that is a finite linear combination of functions of the form $\1_{q\N+r}$, $q\in \N$ and $r\in\{0,1,\ldots,q-1\}$, the product $f\cdot g$ is uniform. Since any Besicovitch rationally almost periodic function can be approximated in the $\|\cdot\|_B$-seminorm by finite linear combinations of functions of the form $\1_{q\N+r}$, it follows from part \ref{itm:prop:interplay-uniform-almost-periodic-a} that for any uniform $f$ and any Besicovitch rationally almost periodic $g$ the function $h=f\cdot g$ is uniform.

Finally, for part \ref{itm:prop:interplay-uniform-almost-periodic-c}, one can easily show by induction that
$$
\|f(tn)\|_{U^{s+1}_{[N]}}^{2^{s+1}}\leq t^{2^{s+1}}\|f\cdot\1_{t\N}\|_{U^{s+1}_{[N]}}^{2^{s+1}}+\oh(N)
$$
and hence the claim follows from part \ref{itm:prop:interplay-uniform-almost-periodic-b0}.
\end{proof}

\section{The dichotomy theorem for $\mconv$}
\label{sec:dt}

In this section we discuss some equivalent characterizations of $\mconv$ and give a proof of \cref{thm:dichotomy}.

\subsection{Equivalent characterizations of $\mconv$}
\label{sec:m_conv}

The following subclass of $\cm$ was introduced in \cref{sec:intro}:
$$
\mconv=\left\{f\in\cm: \lim_{N\to\infty}\frac{1}{N}\sum_{n=1}^N f(qn+r)~\text{exists for all $q,r\in\N$} \right\}.
$$
The next proposition offers an alternative characterizations of functions in $\mconv$. The result must be well-known to aficionados; we provide an elementary proof in the Appendix.

\begin{Proposition}
\label{lem:mconv-characterization}
Let $f\in\cm$. Then $f\in\mconv$ if and only if for all Dirichlet characters $\chi$ the mean value $M(\chi\cdot f)$ exists.
\end{Proposition}

\begin{Remark}\label{rem:c-iii-prime}
Let $f\in\mconv$ and let $\chi$ be a Dirichlet character.
We claim the series $\sum_{p\in\P}\frac{1}{p}(1-f(p)\overline{\chi(p)})$ converges if and only if $\mathbb{D}(f,\chi)<\infty$. To prove this claim it suffices to prove that $\mathbb{D}(f,\chi)<\infty$ implies $\sum_{p\in\P}\frac{1}{p}(1-f(p)\overline{\chi(p)})$ converges, as the other direction is obvious.
Let $q$ denote a modulus of $\chi$. If $q$ is even, then we set $\chi':=\chi$ and if $q$ is odd then we set $\chi':=\chi\cdot \chi_1$, where $\chi_1$ denotes the principal character of modulus $2q$ (cf.\ \eqref{eq:induced-character}).
Since $f\in\mconv$, by \cref{lem:mconv-characterization}, the function $f\cdot\overline{\chi'}$ has a mean.
Therefore $f\cdot\overline{\chi'}$ satisfies either \ref{item:E79-thm6.3-i}, \ref{item:E79-thm6.3-iii} or \ref{item:E79-thm6.3-iv} of \cref{thm:FH-thm.2.9}. However, $f\cdot\overline{\chi'}$ cannot satisfy \ref{item:E79-thm6.3-iii} because $\chi'$ has even modulus and hence $\chi'(2^k)=0$ for all $k$. Also,
$\chi'(p)=\chi(p)$ for all but finitely many primes $p$ and therefore $\mathbb{D}(f,\chi)<\infty$ implies $\mathbb{D}(f,\chi')<\infty$. This implies that
 $f\cdot\overline{\chi'}$ cannot satisfy \ref{item:E79-thm6.3-iv}, because $\mathbb{D}(f\cdot\overline{\chi'},1)=\mathbb{D}(f,\chi')<\infty$.
Therefore $f\cdot\overline{\chi'}$ must satisfy \ref{item:E79-thm6.3-i} of \cref{thm:FH-thm.2.9}, from which it follows that $\sum_{p\in\P}\frac{1}{p}(1-f(p)\overline{\chi(p)})$ converges.
\end{Remark}

Using the above observation we can now replace condition \ref{item:DD82-thm1-and-thm6-iii} in \cref{cor:DD82-thm1-and-thm6} for functions $f\in\mconv$ with a slightly simpler condition (see \hyperref[item:DD82-thm1-and-thm6-iii-prime]{$\text{(iii)}'$} below):

\begin{Corollary}
\label{cor:DD82-thm1-and-thm6-for-mconv}
Let $f\in\mconv$. Then the following conditions are equivalent:
\begin{enumerate}
[label=$\text{(\roman{enumi})}$, ref=$\text{(\roman{enumi})}$, leftmargin=*]
\item\label{item:DD82-thm1-and-thm6-i-prime}
$f$ is Besicovitch almost periodic;
\item\label{item:DD82-thm1-and-thm6-ii-prime}
$f$ is Besicovitch rationally almost periodic;
\item[$\text{(iii)}'$]\label{item:DD82-thm1-and-thm6-iii-prime}
either $\|f\|_B=0$ or there exists a Dirichlet character $\chi$ such that $\mathbb{D}(f,\chi)<\infty$ (in other words, $f$ pretends to be a Dirichlet character).
\end{enumerate}
\end{Corollary}

\begin{Proposition}
\label{cor:rationallity-is-D-transitive}
Let $f,g\in\mconv$ and suppose $\mathbb{D}(f,g)<\infty$. Then $f$ is Besicovitch rationally almost periodic if and only if $g$ is. 
\end{Proposition}

\begin{proof}
Suppose $f$ is Besicovitch rationally almost periodic. We distinguish two cases, the case $\|f\|_B=0$
and the case $\|f\|_B>0$.

If $\|f\|_B=0$ then, by \cref{lem:lH2-C}, we have $\mathbb{D}(|f|,1)=\infty$. It follows from part \ref{item:rem:basic-properties-of-D-5} of \cref{rem:basic-properties-of-D} that
$\mathbb{D}(|f|,|g|)<\infty$ and therefore, using the triangle inequality for $\mathbb{D}(\cdot,\cdot)$, we get $\mathbb{D}(|g|,1)=\infty$. Another application of \cref{lem:lH2-C} shows that $\|g\|_B=0$. Since any function with $\|g\|_B=0$ is trivially Besicovitch rationally almost periodic, this concludes the first case.

Now assume $\|f\|_B>0$. Then, by \cref{thm:DD82-thm6}, there exists a Dirichlet character $\chi$ such that $\sum_{p\in\P}\frac{1}{p}(1-f(p)\overline{\chi(p)})$ converges. This implies that $\mathbb{D}(f,\chi)<\infty$ and, combined with $\mathbb{D}(f,g)<\infty$ and the triangle inequality for $\mathbb{D}(\cdot,\cdot)$, we obtain $\mathbb{D}(g,\chi)<\infty$. Since $g\in\mconv$, we can now use \hyperref[item:DD82-thm1-and-thm6-iii-prime]{$\text{(iii)}'$} from \cref{cor:DD82-thm1-and-thm6-for-mconv} to deduce that $g$ is Besicovitch almost periodic.
\end{proof}

\subsection{Proof of \cref{thm:dichotomy}}
\label{sec:proof-of-dichotomy}

In this subsection we provide a proof of the dichotomy theorem for $\mconv$. The proof is rather short and follows from the results established in the previous subsection and in \cite{DD74,DD82,Delange72, FH17} (see \cref{thm:tAU} and \cref{prop:l10}).

\begin{Lemma}
\label{lem:mconv-orth-archimedean}
Suppose $f\in\mconv$. Then for any Dirichlet character $\chi$ and any $t\in\R\setminus\{0\}$ one has $\mathbb{D}(f,\chi\cdot n^{it})=\infty$.
\end{Lemma} 

\begin{proof}
Suppose there exist a Dirichlet character $\chi$ and some $t\in\R\setminus\{0\}$ such that $\mathbb{D}(f,\chi\cdot n^{it})<\infty$. Let $q$ denote a modulus of $\chi$. If $q$ is even, then we set $\chi':=\chi$ and if $q$ is odd then we set $\chi':=\chi\cdot \chi_1$, where $\chi_1$ denotes the principal character of modulus $2q$.

We now use an argument that has already appeared in \cref{rem:c-iii-prime}.
Since $f\in\mconv$, by \cref{lem:mconv-characterization}, the mean of the function $f\cdot\overline{\chi'}$ exists.
This means that $f\cdot\overline{\chi'}$ satisfies either \ref{item:E79-thm6.3-i}, \ref{item:E79-thm6.3-iii} or \ref{item:E79-thm6.3-iv} of \cref{thm:FH-thm.2.9}. However, $f\cdot\overline{\chi'}$ cannot satisfy \ref{item:E79-thm6.3-iii} because $\chi'$ has even modulus and hence $\chi'(2^k)=0$ for all $k$. 
Since $\chi'(p)=\chi(p)$ for all but finitely many primes $p$, we deduce from $\mathbb{D}(f,\chi\cdot n^{it})<\infty$ that $\mathbb{D}(f\cdot\overline{\chi'},n^{it})<\infty$. It follows that $f\cdot\overline{\chi'}$ cannot satisfy \ref{item:E79-thm6.3-iv}.
Finally, $\mathbb{D}(f\cdot\overline{\chi'},n^{it})<\infty$ together with property \ref{item:rem:basic-properties-of-D-2} listed in \cref{rem:basic-properties-of-D} and \cref{item:rem:basic-properties-of-D-4} imply that $\mathbb{D}(f\cdot\overline{\chi'},1)=\infty$ and therefore $f\cdot\overline{\chi'}$ cannot satisfy \ref{item:E79-thm6.3-i} of \cref{thm:FH-thm.2.9}; we have arrived at a contradiction.
\end{proof}

\begin{proof}[Proof of \cref{thm:dichotomy}]
Let $f\in\mconv$ be arbitrary. If $\mathbb{D}(f,\chi\cdot n^{it})=\infty$ for all $t\in\R$ and all Dirichlet characters $\chi$ then we deduce from \cref{prop:l10} that $f$ is aperiodic and therefore, in view of \cref{thm:tAU}, $f$ is a uniform function.
If, on the other hand, $\mathbb{D}(f,\chi\cdot n^{it})<\infty$ for some $t\in\R$ and some Dirichlet characters $\chi$, then we first apply \cref{lem:mconv-orth-archimedean} to deduce that $t=0$ and hence $\mathbb{D}(f,\chi)<\infty$ and thereafter, using \cref{cor:rationallity-is-D-transitive}, we conclude that $f$ is Besicovitch rationally almost periodic because $\chi$ is periodic.
\end{proof}

\section{The structure theorem for $\Dzero$}
\label{sec:structure-theorem}

The goal of this section is to give a proof of \cref{thm:structure-theorem}. In Subsection \ref{subsec:conditional-uniformity} we discuss in some detail relatively uniform sets. Subsection \ref{sec:52} is devoted to the proof of \cref{thm:structure-theorem} for the special case of level sets of concentrated multiplicative functions. Finally, in Subsection \ref{sec:poF} we establish \cref{thm:structure-theorem} in full generality by reducing it to the special case established in Subsection \ref{sec:52}.

\subsection{Relative uniformity}
\label{sec:rational-sets}
\label{subsec:conditional-uniformity}

In this subsection we provide additional examples of relatively uniform sets and prove a technical lemma which will be needed in the subsequent subsections.

We start with recalling the definition of relative uniformity of sets.
Given sets $E,R\subset\N$ we say \define{$E$ is uniform relative to $R$} if $E\subset R$, $d(E)$ and $d(R)$ exist and the function $d(R)\1_{E}-d(E)\1_{R}$ is uniform, i.e.\ $\|d(R)\1_{E}-d(E)\1_{R}\|_{U^s_{[N]}}$ goes to zero as $N\to\infty$ for all $s\geq 1$.

In \cref{eg:conditional-uniformity-0} we have already seen a natural example of sets $E$ and $R$ such that $E$ is uniform relative to $R$, namely $E=\{n\in\N:\mob(n)=1\}$ and $R=Q$ (where $Q$ denotes the set of squarefree numbers).
We list below some additional examples illustrating relative uniformity.

\begin{Example}
\label{eg:conditional-uniformity}\
\begin{enumerate}
[label=\text{Ex.\ref{eg:conditional-uniformity}.\arabic*:}, ref=\text{Ex.\ref{eg:conditional-uniformity}.\arabic*}, leftmargin=*]
\item

Let $R\subset \N$ be an arbitrary set whose density $d(R)$ exists and is positive. Let $(X_n)_{n\in R}$ be a sequence of $\{0,1\}$-valued independently and identically distributed random variables such that $X_n$ takes on the value $1$ with probability $\frac{1}{2}$ and the value $0$ with probability $\frac{1}{2}$.
It is then straightforward to show (using Hoeffding's inequality) that almost surely the set $E:=\{n\in R: X_n=1\}$ is uniform relative to $R$.

\item
Let $\xi\in[0,1)$ and let $J$ be a Jordan measurable subset of the circle $S^1:=\{w\in\C:|w|=1\}$. It was shown in \cite{FH17-2} that the set $\{n\in\N:\lio_{\xi}(n)\in J\}$ is uniform. It thus follows from \cref{lem:ST-homogeneity-of-uniformity-within-rational-sets} below that the set $E=\{n\in\N:\mob_{\xi}(n)\in J\}$ is uniform relative to the squarefree numbers $Q$, because $Q$ is a rational set and $E=\{n\in\N:\lio_{\xi}(n)\in J\}\cap Q$.
\end{enumerate}
\end{Example}

One can show that if sets $E,R,V\subset\N$ are such that $V$ is rational (see \cref{definition:rational-set}) and $E$ is uniform relative to $R$ then $E\cap V$ is uniform relative to $R\cap V$; in fact we have the following slightly stronger result.

\begin{Lemma}
\label{lem:ST-homogeneity-of-uniformity-within-rational-sets}
Suppose $E\subset R\subset\N$ are sets such that $d(E)$
and $d(R)$ exist and suppose $d(R)\1_E-d(E)\1_R$ is uniform.
Let $t\in\N$, let $V\subset\N$ be any rational set and define $E':=tE\cap V$ and $R':=tR\cap V$.
If $d(R')$ exists, then $d(E')$ exists and satisfies the equation
\begin{equation}
\label{eq:ref-1-1}
d(E)d(R')=d(R)d(E')
\end{equation}
and the function $d(R')\1_{E'}-d(E')\1_{R'}$ is uniform.
\end{Lemma}

\begin{proof}
If $d(R)=0$ then $d(E)=d(E')=d(R')=0$ and hence there is nothing to show. Let us therefore assume that $d(R)> 0$.
Since $d(R)\1_E-d(E)\1_R$ is uniform, it follows from \cref{prop:interplay-uniform-almost-periodic} part \ref{itm:prop:interplay-uniform-almost-periodic-c} that the function 
$d(R)\1_{tE}-d(E)\1_{tR}$ is uniform. Then, using \cref{prop:interplay-uniform-almost-periodic} part \ref{itm:prop:interplay-uniform-almost-periodic-b}, it follows that $(d(R)\1_{tE}-d(E)\1_{tR})\cdot\1_{V}=d(R)\1_{E'}-d(E)\1_{R'}$ is uniform as well.
By definition, any uniform function has zero mean.
From this we immediately obtain the identity $d(E)d(R')=d(R)d(E')$ whenever $d(R')$ exists.
Using this identity and multiplying the function $d(R)\1_{E'}-d(E)\1_{R'}$ by the constant $d(R')/d(R)$ we obtain the function $d(R')\1_{E'}-d(E')\1_{R'}$.
This shows that $d(R')\1_{E'}-d(E')\1_{R'}$ is also uniform.
\end{proof}

\subsection{A proof of \cref{thm:structure-theorem} for the special case of concentrated multiplicative functions}\label{sec:52}

Let $f\colon \N\to\C\setminus\{0\}$ be a concentrated multiplicative function (see \cref{def:concentrated}) and let $\G$ denote its concentration group.
Clearly, $z^{|\G|}=1$ for all $z\in\G$. Let us consider all pairs $(k,\chi)$, where $k\in\N$ and $\chi$
is a Dirichlet character, such that
\begin{equation}
\label{eqn:ST-concentration-character-1}
\mathbb{D}(f^k,\chi)<\infty.
\end{equation}
There is at least one such pair $(k,\chi)$, because we can pick
$k=|\G|$ and $\chi$ to be the principal character of modulus $1$ (i.e.\ $\chi(n)=1$ for all $n\in\N$). This leads to the following definition.

\begin{Definition}
\label{def:ST-concentration-character}
Given a concentrated multiplicative function $f$
with concentration group $\G$ let $k_\G$ denote the smallest
positive integer such that for some Dirichlet character $\chi_\G$
equation \eqref{eqn:ST-concentration-character-1} is satisfied. 
\end{Definition}

The next theorem is a version of \cref{thm:structure-theorem} for concentrated multiplicative functions and constitutes the main result of this subsection. In Subsection \ref{sec:poF} we will show how \cref{thm:structure-theorem} can be derived in its full generality from this special case.

\begin{Theorem}\label{thm:ST}
Let $g$ be a concentrated multiplicative function with concentration group $\G$ and let $k_\G$ be as in \cref{def:ST-concentration-character}.
Then $g^{k_\G}$ is Besicovitch rationally almost periodic and for every $z\in\C\setminus\{0\}$ the set $E_g:=E(g,z)$ is uniform relative to $R_{g}:=E(g^{k_\G},z^{k_\G})$.
\end{Theorem}

For the proof of \cref{thm:ST} we need three lemmas.

\begin{Lemma}
\label{lem:lifting-trick-2}
Let $p\in\P$ and let $k,m\in\N$ and let $c\geq 1$. Let $f$ and $g$ be multiplicative functions and suppose $f(q^\ell)=g(q^\ell)$ for all pairs $(q,\ell)\in\P\times \N$ with $(q,\ell)\neq (p,k)$.
Assume $f^m$ is Besicovitch rationally almost periodic and for every $z\in\C\setminus\{0\}$ the set $E_f:=E(f,z)$ is uniform relative to $R_{f}:=E(f^m,z^m)$ and $c d(E_f)=d(R_f)$.
Then $g^m$ is Besicovitch rationally almost periodic and for every $z\in\C\setminus\{0\}$ the set $E_g:=E(g,z)$ is uniform relative to $R_{g}:=E(g^m,z^m)$ and $c d(E_g)=d(R_g)$.
\end{Lemma}

\begin{proof}
Let $z\in\C\setminus\{0\}$ be arbitrary. Let
\begin{equation}
\label{eqn:T-11}
\begin{split}
T:=&~\{n\in\N:\text{$n=s\cdot p^k$ for some $s\in\N$ with $\gcd(s,p)=1$}\}
\\
=&~\bigcup_{a=1}^{p-1}p^{k}((p\N\cup\{0\})+a)
\end{split}
\end{equation}
and
\begin{equation}
\label{eqn:S-11}
S:=\N\setminus T.
\end{equation}
Note that $S$ is a multiplicative set.
Clearly,
\begin{equation}
\label{eq:intst-11}
E_g\cap S = E_f\cap S
\qquad\text{and}\qquad
R_g\cap S = R_f\cap S.
\end{equation}

Define
$$
E_f':=
\begin{cases}
\left\{n\in\N: f(n)= \tfrac{z f(p^k)}{g(p^k)}\right\},&\text{if}~g(p^k)\neq 0;
\\
\emptyset,&\text{if $g(p^k)=0$}
\end{cases}
$$
and
$$
R_f':=
\begin{cases}
\left\{n\in\N: f^m(n)= \left(\tfrac{z f(p^k)}{g(p^k)}\right)^m\right\},&\text{if}~g(p^k)\neq 0;
\\
\emptyset,&\text{if $g(p^k)=0$.}
\end{cases}
$$
If $g(p^k)\neq 0$ then, by assumption, $E'_f$ is uniform relative to $R'_f$ and $cd(E'_f)=d(R'_f)$.
On the other hand, if $g(p^k)=0$ then $E'_f=R'_f=\emptyset$ and hence it is trivially satisfied that $E'_f$ is uniform relative to $R'_f$ and $cd(E'_f)=d(R'_f)$.

Let $n\in T$ be arbitrary and write $n=s\cdot p^k$ with $\gcd(p,s)=1$.
If $g(p^k)\neq 0$ then
$$
g(n)=z~\Leftrightarrow~g(s) = \tfrac{z}{g(p^k)}~\Leftrightarrow~f(s) = \tfrac{z}{g(p^k)}~\Leftrightarrow~f(n) = \tfrac{zf(p^k)}{g(p^k)}.
$$
If $g(p^k)= 0$ then $g(n)=z$ holds for no $n\in T$, because $z\neq 0$.
This proves that
\begin{equation}
\label{eq:intst-12}
E_g\cap T =  E_f'\cap T.
\end{equation}
An analogous calculation shows that 
\begin{equation}
\label{eq:intst-13}
R_g\cap T = R_f'\cap T.
\end{equation}

Combining \eqref{eq:intst-11}, \eqref{eq:intst-12} and \eqref{eq:intst-13} we obtain
\begin{eqnarray}
\label{eq:s121}
E_{g}&=&\Big(E_{f}\cap S\Big)\cup\Big(E_{f}'\cap T\Big),
\\
\label{eq:s122}
R_{g}&=&\Big(R_{f}\cap S\Big)\cup\Big(R_{f}'\cap T\Big).
\end{eqnarray}

Our goal is to show that the function $d(R_{g})\1_{ E_{g}}-d(E_{g})\1_{ R_{g}}$ is uniform. It follows from \cref{cor:multiplicative-fibers-have-density} that the density of $R_g$ exists. If $d(R_g)=0$ then the $\Vert\cdot\Vert_B$-norm of $d(R_{g})\1_{ E_{g}}-d(E_{g})\1_{ R_{g}}$ equals $0$ and hence this function is uniform for trivial reasons. We can therefore assume without loss of generality that $d(R_g)>0$.

Since $\1_S$ is a $\{0,1\}$-valued multiplicative function, we deduce from \cref{ren:0-1-valued-RAP} that $S$ is a rational set.
Moreover, $R_{f}\cap S=\{n\in \N: f^m(n)\1_S(n)=z^m\}$ and therefore the density $d\big(R_{f}\cap S\big)$ exists by \cref{cor:multiplicative-fibers-have-density}.
Similarly $d(E_f\cap S)$ exists. Now, by \eqref{eq:intst-11} and \cref{lem:ST-homogeneity-of-uniformity-within-rational-sets} (applied to $E_f\subset R_f$ and $S$), we obtain that the function
\begin{equation*}
d(R_{g}\cap S)\1_{E_{f}\cap S}-d(E_{g}\cap S)\1_{R_{f}\cap S}=\frac{d(R_{g}\cap S)}{d(R_{g})}\Big(d(R_{g})\1_{E_{f}\cap S}-d(E_{g})\1_{R_{f}\cap S}\Big)
\end{equation*}
is uniform.
From this we conclude that
\begin{equation}
\label{eqn:uf-S}
\Big(d(R_{g})\1_{E_{f}}-d(E_{g})\1_{R_{f}}\Big)\cdot\1_S
\end{equation}
is also uniform. Also, from \eqref{eq:ref-1-1} and $d(R_{f})=cd(E_{f})$ we get $d\big(R_{f}\cap S\big)=cd\big(E_{f}\cap S\big)$.

Analogous to the way we proved that $d(R_f\cap S)$ exists, one can show that $d(R_f'\cap S)$ exists. It follows that $d(R_f'\cap T)=d\big(R_f'\setminus(R_f'\cap S)\big)=d(R_f')-d(R_f'\cap S)$ also exists. Additionally, since $S$ is rational, the set $\N\setminus S=T$ is rational.
Using the fact that $E'_f$ is uniform relative to $R'_f$ together with \eqref{eq:intst-12}, \eqref{eq:intst-13} and \cref{lem:ST-homogeneity-of-uniformity-within-rational-sets} (applied to $E_f'\subset R_f'$ and $T$) we deduce that $d\big(E_{f}'\cap T\big)$ exists and that
\begin{equation}
\label{eqn:uf-T}
\Big(d(R_{g})\1_{ E_{f}'}-d(E_{g})\1_{R_{f}'}\Big)\cdot\1_T
\end{equation}
is uniform. From \eqref{eq:ref-1-1} and $d(R_{f}')=cd(E_{f}')$ we obtain $d\big(R_{f}'\cap T\big)=cd\big(E_{f}'\cap T\big)$.

Since the sum of two uniform functions remains uniform (due to the triangle inequality for $\|\cdot\|_{U_{[N]}^s}$), we conclude by taking the sum of \eqref{eqn:uf-S} and \eqref{eqn:uf-T} and utilizing \eqref{eq:s121} and \eqref{eq:s122} that $d(R_{g})\1_{ E_{g}}-d(E_{g})\1_{ R_{g}}$ is uniform. Moreover, combining $d\big(R_{f}'\cap T\big)=cd\big(E_{f}'\cap T\big)$ and $d\big(R_{f}\cap S\big)=cd\big(E_{f}\cap S\big)$ with \eqref{eq:s121} and \eqref{eq:s122} we obtain $c d(E_g)=d(R_g)$.

It is straightforward to show that if $h$ is a Besicovitch rationally almost periodic function then for any $q\in \N$ so is
$$
h_0(n):=
\begin{cases}
h\left(\frac{n}{q}\right),&\text{if}~q\mid n
\\
0,&\text{otherwise}.
\end{cases}
$$
In particular, the function
$$
h_1(n):=
\begin{cases}
f^m\left(\frac{n}{p^k}\right),&\text{if}~p^k\mid n
\\
0,&\text{otherwise}
\end{cases}
$$
is Besicovitch rationally almost periodic.
Since $S$ and $T$ are rational sets, it follows that the functions $f^m\cdot\1_S$ and $h_1\cdot\1_{T}$ are Besicovitch rationally almost periodic. Note that any $n\in T$ satisfies $p^k\mid n$. Hence,
$$
h_3(n):=g^m(p^k) h_1\cdot\1_{T}=
\begin{cases}
g^m(p^k) f^m\left(\frac{n}{p^k}\right),&\text{if}~n\in T;
\\
0,&\text{otherwise},
\end{cases}
$$
is Besicovitch rationally almost periodic.
Therefore $g^m=f^m\cdot\1_S+h_3$ is Besicovitch rationally almost periodic.
\end{proof}

\begin{Lemma}
\label{lem:lifting-trick-3}
Let $P\subset\P$ with $\sum_{p\in\P\setminus P}\tfrac{1}{p}<\infty$ and let $m\in\N$ and $c\geq 1$. Let $f$ and $g$ be multiplicative functions and suppose $f(p)=g(p)$ for all $p\in\P\setminus P$.
Assume $f^m$ is Besicovitch rationally almost periodic and for every $z\in\C\setminus\{0\}$ the set $E_f:=E(f,z)$ is uniform relative to $R_{f}:=E(f^m,z^m)$ and $c d(E_f)=d(R_f)$.
Then $g^m$ is Besicovitch rationally almost periodic and for every $z\in\C\setminus\{0\}$ the set $E_g:=E(g,z)$ is uniform relative to $R_{g}:=E(g^m,z^m)$.
\end{Lemma}

\begin{proof}
Let $\Omega:=\{(p,k)\in\P\times \N: f(p^k)\neq g(p^k)\}$.
Note that $\Omega$ can be turned into a linearly ordered set $(\Omega,\prec)$ using the relation
$$
(p,k)\prec (q,\ell)\quad\Leftrightarrow\quad p^k<q^\ell.
$$
Let $(p_1,k_1)\prec (p_2,k_2)\prec\dots$ be an enumeration of $\Omega$.

We now define inductively a sequences of multiplicative functions $f_0,f_1,f_2,\ldots$ as follows. First, we let $f_0:=f$; then we define
$$
f_{i+1}(p^k):=
\begin{cases}
f_i(p^k),&\text{if}~(p,k)\neq (p_{i+1},k_{i+1});
\\
g(p^k),&\text{otherwise}.
\end{cases}
$$
Note that for a fixed $n\in\N$ there exists $i_n$ such that $f_i(n)=g(n)$ for all $i\geq i_n$. 

Since $\sum_{p\in\P\setminus P}\tfrac{1}{p}<\infty$, it follows that
$$
\sum_{(p,k)\in\Omega}\tfrac{1}{p^k}<\infty.
$$
Also,
\begin{multline}
\label{eqn:d-hi-f-0}
\overline{d}\Big(\big\{n\in\N: g^m(n)\neq f^m_i(n)\big\}\Big)\\
\leq\overline{d}\Big(\big\{n\in\N: g(n)\neq f_i(n)\big\}\Big)
\leq\overline{d}\left(\bigcup_{(p,k)\in\Omega\atop (p_i,k_i)\prec(p,k)} p^k\N\right)\leq \sum_{(p,k)\in\Omega\atop (p_i,k_i)\prec(p,k)}\tfrac{1}{p^k}.
\end{multline}
It follows that $\lim_{i\to\infty}\|g-f_i\|_B=0$ and $\lim_{i\to\infty}\|g^m-f_i^m\|_B=0$.

Let $z\in \C\setminus\{0\}$ be arbitrary.
Recall that, by assumption, $E_f$ is uniform relative to $R_f$. Define $E_{f_i}:=E(f_i, z)$ and $R_{f_i}:=E(f_i^m,z^m)$. It clearly follows from \cref{lem:lifting-trick-2} and induction on $i$ that $f_i^m$ is Besicovitch almost periodic, $E_{f_i}$ is uniform relative to $R_{f_i}$ and $cd(E_{f_i})=d(R_{f_i})$. Therefore, $g^m$ is Besicovitch almost periodic, because $\lim_{i\to\infty}\|g^m-f_i^m\|_B=0$.
We deduce from \eqref{eqn:d-hi-f-0} that
\begin{equation}
\label{eqn:d-hi-f-3}
\lim_{i\to\infty}\overline{d}\big(E_{g}\triangle E_{f_i}\big)=0
\qquad\text{and}\qquad
\lim_{i\to\infty}\overline{d}\big(R_{g}\triangle R_{f_i}\big)=0,
\end{equation}
where $E_{g}:=E(g,z)$ and $R_{g}:=E(g^m,z^m)$.
Hence
$$
\left\|\Big(d(R_{g})\1_{E_{g}}-d(E_{g})\1_{R_{g}}\Big)-\Big(d(R_{f_i})\1_{E_{f_i}}-d(E_{f_i})\1_{R_{f_i}}\Big)\right\|_B\xrightarrow[]{i\to\infty}0.
$$
Finally, using part \ref{itm:prop:interplay-uniform-almost-periodic-a} of \cref{prop:interplay-uniform-almost-periodic} we deduce that $E_{g}$ is uniform relative to $R_{g}$. This finishes the proof.
\end{proof}

\begin{Lemma}
\label{lem:f-to-the-k-is-chi}
Let $m\in\N$, $f$ a multiplicative function and $\chi$ a Dirichlet character. Assume that $f^j$ is aperiodic for all $j\in\{1,2,\ldots,m-1\}$ and that $f^m=\chi$. Let $z\in\C$ and set $E:=E(f,z)$ and $R:=E(\chi,z^m)$. Then $E$ is uniform relative to $R$ and $m d(E)=d(R)$. 
\end{Lemma}

\begin{proof}
First, using \cref{thm:tAU}, we deduce that for all $j\in\{1,2,\ldots,m-1\}$ the function $f^j$ is uniform.
Also, note that the density of $E$ and $R$ exists, due to \cref{cor:multiplicative-fibers-have-density}.
It remains to show that the function
\begin{equation}
\label{eqn:ST-proof-of-ST-2-old}
d\left(R\right)\1_{E}-d\left(E\right)\1_{R}
\end{equation}
is uniform.

If $z=0$ then $R=E$ (because $f^m=\chi$) and so the function $d\left(R\right)\1_{E}-d\left(E\right)\1_{R}$ is constant $0$ and hence uniform.
We can therefore assume without loss of generality that $z\neq 0$.

By assumption, for any $n\in R$ we have $f^m(n)=\chi(n)=z^m$. 
Therefore, the number $z^{-1}f(n)$ is an $m$-th root of unity for any $n\in R$. It follows that for all $n\in R$,
\[
\frac{1}{m}\sum_{j=0}^{m-1} z^{-j}f^j(n)=
\begin{cases}
1,& \text{if $f(n)=z$};\\
0,& \text{otherwise}.
\end{cases}
\]
So,
\[
\1_{E}= \1_{R}\cdot
\left(\frac{1}{m}\sum_{j=0}^{m-1} z^{-j}f^j \right)
\]
and after rearranging we get 
\begin{equation}
\label{eqn:ST-proof-of-ST-3-old}
\1_{E}- \frac{1}{m}\1_{R}= \1_{R}\cdot
\left(\frac{1}{m}\sum_{j=1}^{m-1} z^{-j}f^j \right).
\end{equation}

Since $\1_R$ is Besicovitch rationally almost periodic and $f^j$ is uniform for $j=1,...,m-1$, by \cref{prop:interplay-uniform-almost-periodic} \ref{itm:prop:interplay-uniform-almost-periodic-b}, we deduce that the right hand side of \eqref{eqn:ST-proof-of-ST-3-old} is uniform.
This implies that
\begin{equation}
\label{eqn:ST-proof-of-ST-8-old}
\1_{E}- \frac{1}{m}\1_{R}
\end{equation}
is uniform as well.
Since any uniform function has zero mean, it follows that
$d(E)m=d\left(R\right)$ and so the function in
\eqref{eqn:ST-proof-of-ST-2-old} is a constant multiple of the function in \eqref{eqn:ST-proof-of-ST-8-old} and hence also uniform. 
\end{proof}


\begin{proof}[Proof of \cref{thm:ST}]
Let  $\G$ denote the concentration group of $g$ and let 
$k_\G$ and $\chi_\G$ be as in \cref{def:ST-concentration-character}.
Define $\Omega_\G:=\{(p,k)\in\P\times\N: g(p^k)\in\G,~g^{k_\G}(p^k)=\chi_\G(p^k)\}$.
Since the pair $(k_\G,\chi_\G)$ satisfies \eqref{eqn:ST-concentration-character-1}, we have that
$
\sum_{(p,k)\notin \Omega_\G}\frac1{p^k}<\infty.
$

Given $(p,k)\notin\Omega_\G$ let $\xi_{(p,k)}$ be any complex number that satisfies $\xi_{(p,k)}^{k_\G}=\chi_\G(p^k)$. 
Define a new multiplicative function $f$ via
$$
f(p^k):=
\begin{cases}
g(p^k), &\text{if}~(p,k)\in\Omega_\G;
\\
\xi_{(p,k)},&\text{otherwise}.
\end{cases}
$$
Note that $f$ satisfies the functional equation
\begin{equation}
\label{eqn:ST-proof-of-ST-1}
f^{k_\G}=\chi_\G.
\end{equation}

We claim that $\mathbb{D}(f^j,\chi\cdot n^{it})=\infty$ for all $j\in\{1,2,\ldots,k_\G-1\}$,
for all $t\in\R$ and for all Dirichlet characters $\chi$. To verify this claim
we have to distinguish between the case $t=0$ and the case $t\in\R\setminus\{0\}$.

The case $t=0$ follows from the minimality assumption on $k_\G$: $\mathbb{D}(g^j,\chi)=\infty$ for each $j=1,...,k_G-1$ and each Dirichlet character $\chi$. Since $\sum_{p\in\P\atop f(p)\neq g(p)}\frac{1}{p}<\infty$, it follows from the triangle inequality for $\mathbb{D}$ that $\mathbb{D}(f^j,\chi)=\infty$ for each $j=1,...,k_G-1$ and each Dirichlet character $\chi$.

For the case $t\neq 0$ we give a proof by contradiction. Let us assume that there are $j\in\{1,\ldots,k_\G\}$, a Dirichlet character $\chi$ and a number 
$t\in\R\setminus\{0\}$ such that $\mathbb{D}(f^j,\chi\cdot n^{it})<\infty$.
Using part \ref{item:rem:basic-properties-of-D-3} of \cref{rem:basic-properties-of-D} it follows that also
$\mathbb{D}(f^{j|\G|},\chi^{|\G|} \cdot n^{it|\G|})<\infty$.
However, for all primes $p$ with $g(p)\in \G$ we have that $g^{j|\G|}(p)=f^{j|\G|}(p)=1$. Hence, $\mathbb{D}(f^{j|\G|},\chi^{|\G|} \cdot n^{it|\G|})<\infty$ implies $\mathbb{D}(\overline{\chi}^{|\G|},n^{it|\G|})<\infty$. This contradicts the statement of \cref{DirArch}.

Since $\mathbb{D}(f^j,\chi\cdot n^{it})=\infty$ for all
$j\in\{1,2,\ldots,k_\G-1\}$, all $t\in\R$ and all Dirichlet
characters $\chi$, it follows from \cref{prop:l10} that $f^j$ is aperiodic. It therefore follows from \cref{lem:f-to-the-k-is-chi} that for all $z\in \C\setminus\{0\}$ the set $E_{f}:=E(f,z)$ is uniform relative to $R_{f}:=E(\chi_\G,z^{k_\G})$ and $k_\G d(E_f)=d(R_f)$.

Finally, observe that $f$ and $g$ are two multiplicative functions that satisfy the conditions of \cref{lem:lifting-trick-3} (with $c=m=k_\G$), from which we conclude that $g^{k_\G}$ is Besicovitch rationally almost periodic and for every $z\in \C\setminus\{0\}$ the set $E_{g}:=E(g,z)$ is uniform relative to $R_{g}:=E(g^{k_\G},z^{k_\G})$.
\end{proof}

\subsection{A proof of \cref{thm:structure-theorem}}
\label{sec:poF}

In this subsection we give a proof of \cref{thm:structure-theorem}.
The proof is based on the idea that any multiplicative function $f$ either behaves like a concentrated multiplicative function, in which case \cref{thm:structure-theorem} can be derived from \cref{thm:ST}, or all sets of the form $E:=\{n\in\N:f(n)=z\}$ with $z\neq 0$ have zero density.
This only leaves the case $z=0$, which can be taken care of by using the characterization of Besicovitch rationally almost periodic multiplicative functions due to Daboussi and Delange discussed in Subsection~\ref{subsec:AP}.

We will need the following lemma.

\begin{Lemma}
\label{lem:common-unit-value}
Suppose $E_1,E_2\in\Dzero$ and $0<d(E_1),d(E_2)<1$. Then $d(E_1\triangle E_2)=0$ if and only if $E_1=E_2$.\footnote{Note that if $d(E_1)=d(E_2)=0$ or $d(E_1)=d(E_2)=1$ then $d(E_1\triangle E_2)=0$ does not necessarily imply $E_1=E_2$. Take for instance $E_1=\{1,2\}$ and $E_2=\{1,3\}$ or $E_1=\N\setminus\{1,2\}$ and $E_2=\N\setminus\{1,3\}$, which are sets belonging to $\Dzero$ because the functions $\1_{\{1,2\}}$ and $\1_{\{1,3\}}$ are multiplicative.}
\end{Lemma}

\begin{proof}
Clearly $E_1=E_2$ implies $d(E_1\triangle E_2)=0$. To prove the other direction we assume that there exists $n_0\in E_1$ with $n_0\notin E_2$ and show that this leads to a contradiction with $d(E_1\triangle E_2)=0$.

By definition of $\Dzero$ there exist multiplicative functions $f_1,f_2\colon \N\to\C$ and numbers $z_1,z_2\in\C$ such that $E_1=E(f_1,z_1)$ and $E_2=E(f_2,z_2)$.
We have to distinguish three cases, the case $z_1=z_2=0$, the case $z_1\neq0$ and $z_2\neq 0$ and finally the case $z_1=0$ and $z_2\neq 0$. We remark that the case $z_1\neq0$ and $z_2= 0$ is analogous to the case $z_1=0$ and $z_2\neq 0$ and is therefore omitted.

If $z_1=z_2=0$ then for $i\in\{1,2\}$ we define $g_i(n)=0$ if $f_i(n)=0$ and $g_i(n)=1$ if $f_i(n)\neq 0$. It is clear that $g_i=\1_{\N\setminus E_i}$ and $E_i=E(g_i,0)$. Since $d(E_i)<1$, we have that $\|g_i\|_B>0$ and hence, in view of \cref{lem:lH2-C}, the sets $P_i:=\{p\in\P: g_i(p)=1\}$ satisfy $\sum_{p\in \P\setminus P_i}\frac{1}{p}<\infty$.
Let $P$ denote the set of all primes that belong to both $P_1$ and $P_2$ and that do not divide $n_0$. Let $S_P\subset \N$ be defined as
\begin{equation}\label{eqn:S_P}
S_P:=\left\{n\in\N:\text{there exist distinct $p_1,\ldots, p_t\in P$ such that $n=p_1\cdot\ldots\cdot p_t$}\right\}.
\end{equation}
Then by \cref{lem:lH2-C} we have $d(S_P)>0$.
Since $n_0\in E_1$ but $n_0\notin E_2$ and $n_0$ is coprime to all numbers in $S_P$, it follows that $E_1\setminus E_2$ contains the set $n_0 S_P$. In particular, $d(E_1\setminus E_2)\geq d(n_0S_P)>0$. This, however, contradicts $d(E_1\triangle E_2)=0$.

Next, assume $z_1\neq0$ and $z_2\neq 0$. Using \cref{cor:density->concentrated function} we can find two concentrated multiplicative functions $g_1,g_2\colon \N\to\C\setminus\{0\}$ such that $E_1=E(g_1,z_1)$ and $E_2=E(g_2,z_2)$. Define $\vec{g}:=(g_1,g_2)$ and let $im(\vec{g})\subset(\C\setminus\{0\})^2$ denote the image of $\vec{g}$.
Since $g_1$ and $g_2$ are concentrated multiplicative functions, also $\vec{g}$ is concentrated, see Remark~\ref{vector}. We now use an argument similar to the one used in the proof of \cref{cor:density->concentrated function}. Choose $\vec{y}\in (\C\setminus\{0\})^2$ such that $(\vec{y}^n\cdot im(\vec{g}))\cap im(\vec{g})=\emptyset$ for all $n\in\N$. We define a new multiplicative function $\vec{h}=(h_1,h_2)$ via
$$
\vec{h}(p^k):=
\begin{cases}
\vec{g}(p^k),&\text{if } p \nmid n_0
\\
\vec{y},&\text{if }p \mid n_0
\end{cases},
\qquad\forall k\in\N,~\forall p\in\P.
$$
It is straightforward to verify that $\vec{g}(n)=\vec{h}(n)$ if and only if $\gcd(n,n_0)=1$ and $\vec{h}(n)\notin im(\vec{g(n)})$ for all $n$ with $\gcd(n,n_0)>1$. Since $\vec{g}$ satisfies \ref{itm_a}, \ref{itm_b} and \ref{itm_c} of \cref{thm:ST-Ruzsa-3.10}, also $\vec{h}$ satisfies them because the number of primes $p$ for which $\vec{g}(p)\neq \vec{h}(p)$ is finite. Thus, $\vec{h}$ is concentrated, whence the set $E(\vec{h},(1,1))=\{n\in\N: h_1(n)=1~\text{and}~h_2(n)=1\}$ has positive density by Remark~\ref{vector}. Note that $\vec{h}(n)=(1,1)$ if and only if $\vec{g}(n)=(1,1)$ and $\gcd(n,n_0)=1$. Hence
$$
E(\vec{h},(1,1))=\{n\in\N: g_1(n)=1,~g_2(n)=1,~\gcd(n,n_0)=1\}.
$$
We obtain that $g_1(n_0 m)=g_1(n_0)$ and $g_2(n_0 m)=g_2(n_0)$ for all $m\in E(\vec{h},(1,1))$. In particular $n_0 E(\vec{h},(1,1))\subset E_1\setminus E_2$, which contradicts $d(E_1\triangle E_2)=0$.

Finally, we deal with the case $z_1=0$ and $z_2\neq 0$. Let $g$ denote the multiplicative function defined as $g_1(n)=0$ if $f_1(n)=0$ and $g_1(n)=1$ if $f_1(n)\neq 0$.
Let $P:=\{p\in\P: p\nmid n_0,~g_1(p)=1\}$ and let $S_P\subset \N$ be defined as in \eqref{eqn:S_P}. Arguing as in the case $z_1=z_2=0$ above one can show that $d(S_P)>0$.
Next, using \cref{cor:density->concentrated function}, we can find a concentrated multiplicative function $g_2\colon \N\to\C\setminus\{0\}$ such that $E_2=E(g_2,z_2)$.
Then, using arguments similar to the ones utilized in the previous paragraph, we first find $y\in \C\setminus\{0\}$ such that $(y^n\cdot im(g_2))\cap im(g_2)=\emptyset$ for all $n\in\N$ and then define a multiplicative function $h\colon\N\to\C\setminus\{0\}$ via
$$
h(p^k):=
\begin{cases}
g_2(p),&\text{if } p \in P~\text{and }k=1
\\
\vec{y},&\text{if either }p \notin P~\text{or }k\geq2
\end{cases},
\qquad\forall k\in\N,~\forall p\in\P.
$$
It is straightforward to verify that $g_2(n)=h(n)$ if and only if $n\in S_P$ and $h(n)\notin im(g_2(n))$ for all $n$ which are either not squarefree or satisfy $p\mid n$ for some $p\in \P\setminus P$. Since $g_2$ is concentrated and $\sum_{p\in\P\setminus P}\frac{1}{p}<\infty$, $h$ is concentrated too. It follows from \cref{thm:ST-Ruzsa-3.10} that $E(h,1)$ has positive density.
Since $h(n)=1$ if and only if $g(n)=1$ and $n\in S_P$, we obtain that $g_1(n_0 m)=g_1(n_0)$ and $g_2(n_0 m)=g_2(n_0)$ for all $m\in E(h,1)\subset S_P$. In particular, $n_0 E(h,1)\subset E_1\setminus E_2$, which again contradicts $d(E_1\triangle E_2)=0$.
\end{proof}

\begin{proof}[Proof of \cref{thm:structure-theorem}]
Let $E\in\Dzero$ and suppose $d(E)>0$.
By definition there exists a multiplicative function $f$ such that $E=E(f,z)$.
Our goal is to find a set $R\in\Drat$ such that $E$ is uniform relative to $R$.
We distinguish two cases, $z=0$ and $z\neq 0$.

If $z=0$ then let $g$ be the multiplicative function defined as $g(n)=0$ if $f(n)=0$ and $g(n)=1$ if $f(n)\neq 0$. In view of \cref{ren:0-1-valued-RAP}, $g$ is Besicovitch rationally almost periodic. Also, $E=E(f,z)=E(g,z)$, which proves that the set $E$ belongs to $\Drat$. Since any set is uniform relative to itself, we can simply pick $E=R$ and are done.
 
Now assume $z\neq 0$. Using \cref{cor:density->concentrated function} we can find a concentrated multiplicative function $g\colon \N\to\C\setminus\{0\}$ such that $E=E(f,z)=E(g,z)$. According to \cref{thm:ST} there exist a Besicovitch rationally almost periodic multiplicative function $h$ and $y\in \C\setminus\{0\}$ such that $E$ is uniform relative to $R:=E(h,y)$ (namely $h=g^{k_\G}$ and $y=z^{k_\G}$). Clearly, the set $R$ belongs to $\Drat$. This proves the claim.

Finally, we have to show that if $0<d(E)<1$ then the set $R\in\Drat$ such that $E$ is uniform relative to $R$ is unique. Suppose $R'\in \Drat$ is another set such that $E$ is uniform relative to $R'$.
Since $\1_{R'}$ is Besicovitch rationally almost periodic and $d\left(R\right)\1_{E}-d\left(E\right)\1_{R}$ is uniform, it follows from part \ref{itm:prop:interplay-uniform-almost-periodic-b} of \cref{prop:interplay-uniform-almost-periodic} that the function
\begin{equation}
\label{eq:r-r-prime}
\left(d\left(R\right)\1_{E}-d\left(E\right)\1_{R}\right)\cdot \1_{R'}=
d\left(R\right)\1_{E}-d\left(E\right)\1_{R\cap R'}
\end{equation}
is uniform. Since any uniform function has zero mean, we have that
$$
\lim_{N\to\infty}\frac{1}{N}\sum_{n=1}^N d\left(R\right)\1_{E}(n)-d\left(E\right)\1_{R\cap R'}(n)=0,
$$
which shows that $d(R)=d(R\cap R')$. By symmetry, it follows that $d(R)=d(R\cap R')=d(R')$ and hence $d(R\triangle R')=0$. In view of \cref{lem:common-unit-value}, this proves that $R=R'$.
\end{proof}

\begin{Remark}
It is natural to wonder if \cref{thm:structure-theorem} extends to the sets of the form  $\{n\in\N: f\in I\}$, where $f$ is a multiplicative function taking values in the unit circle and $I$ is an arc. However, the density of such sets need not exist (take for instance $f(n)=n^{it}$ for some $t\in\R\setminus\{0\}$ and $I$ to be any arc that is not the full circle).
To avoid dealing with issues of this kind one can switch to working with the somewhat weaker but universal notion of logarithmic density. But one still faces the problem that no analogue of \cref{thm:ST-Ruzsa-3.10} seems to exists in this set-up. Thus formulating and proving an appropriate analogue of \cref{thm:structure-theorem} for this wider class of sets appears to be a non-trivial task.
\end{Remark}

\section{Applications to Ergodic Theory and Combinatorics}
\label{sec:ET}

In this section we provide proofs of \cref{thm:recurrence-divisibility} and \cref{thm:recurrence-enhanced}.

\subsection{The class $\Drat$}
\label{sec:D_rat}

In this subsection we prove some basic facts about elements in $\Drat$ (see \cref{def:c0rat}); these properties will be needed for the proofs of \cref{thm:recurrence-divisibility} and \cref{thm:recurrence-enhanced} in the next subsection.

Given a set $E\subset\N$ consider the following two conditions:
\begin{enumerate}
[label=$\text{(\Alph{enumi})}$, ref=$\text{(\Alph{enumi})}$, leftmargin=*]
\item\label{itm:Drat-1}
$E$ is a rational set;
\item\label{itm:Drat-2}
for all $q\in\N$ and all $r\in\{0, 1,\ldots,q-1\}$ either $E\cap (q\N-r)=\emptyset$ or $d(E\cap(q\N-r))$ exists and is positive.
\end{enumerate}

\begin{Lemma}\label{lem:ST-zeros-of-f-form-rational-set}
If $f\colon \N\to\C$ is a multiplicative function and $0$ lies in the image of $f$, then the level set $T:=E(f,0)$ satisfies conditions \ref{itm:Drat-1} and \ref{itm:Drat-2}.
\end{Lemma}

\begin{proof}
Let $g(n)$ be the multiplicative function defined as $g(n)=0$ if $f(n)=0$ and $g(n)=1$ if $f(n)\neq 0$. Then $T=E(g,0)$. However, using \cref{lem:lH2-C}, we either have $\|g\|_B=0$ or $\mathbb{D}(g,1)<\infty$.
If $\|g\|_B=0$, then $d(T)=1$, which implies that $T$ is rational. On the other hand, if $\mathbb{D}(g,1)<\infty$ then $\mathbb{D}(g,1)=\sum_{p\in\P}\tfrac{1}{p}(1-g(p))<\infty$ and therefore, using \cref{cor:DD82-thm1-and-thm6}, we deduce that $g$ is Besicovitch rationally almost periodic, which implies that $T$ is rational.
This shows that $T$ satisfies \ref{itm:Drat-1}.

Next, let $q\in\N$ and $r\in\{0, 1,\ldots,q-1\}$. Since $T$ is a rational set, the density $d(T\cap(q\N-r))$ exists. It remains to show that if $T\cap (q\N-r)\neq\emptyset$ then $d(T\cap(q\N-r))$ is positive. Suppose $x\in T\cap (q\N-r)$. Let $S:=\{n\in\N: \gcd(x,n)=1\}$. Then $xS\subset T$. Also, $xS$ is a finite union of infinite arithmetic progressions and hence $xS\cap(q\N-r)$ is a non-empty finite union of infinite arithmetic progressions. This shows that $d(xS\cap(q\N-r))$ exists and is positive, which, in turn, proves that $T$ satisfies \ref{itm:Drat-2}.
\end{proof}

\begin{Proposition}
\label{prop:Drat}
Suppose $R\in\Drat$ and $d(R)>0$. Then $R$ satisfies conditions \ref{itm:Drat-1} and \ref{itm:Drat-2}.
\end{Proposition}

\begin{Remark}
\cref{prop:Drat} implies that any level set of a Besicovitch almost periodic multiplicative function is rational. This fails to be true for general (not necessarily multiplicative) Besicovitch rationally almost periodic functions. Indeed, let $D\subset\N$ be arbitrary and consider the function $f\colon\N\to\C$ defined as $f(n)=0$ if $n\in D$ and $f(n)=\frac{1}{n}$ if $n\notin D$. Then $f$ is Besicovitch almost periodic, because $\|f\|_B=0$, and $E(f,0)=D$. This shows that any set whatsoever, and in particular any non-rational set, can be realized as a level set of a Besicovitch almost periodic function.
\end{Remark}

Before proving \cref{prop:Drat} we recall the definition of inner regular sets.

\begin{Definition}[see {\cite[Definition 2.3]{BR02}} and \cite{BKLR16arXiv}]
\label{def:inner-reg}
A subset $R\subset\N$ is called \define{inner regular} if for each $\epsilon>0$ there exists $m\in\N$ such that for each $s\in \{0,1,\ldots,m-1\}$ the intersection $R\cap (m\N-s)$ is either empty or has lower density $>\frac{1-\epsilon}{m}$.
\end{Definition}

\begin{Remark}\label{rem:inner-reg-AB}
It follows immediately from \cref{def:inner-reg} that any inner regular set satisfies condition \ref{itm:Drat-1}. We claim that inner regular sets also satisfy condition \ref{itm:Drat-2}. To prove this claim, let $q\in\N$ and $r\in\{0, 1,\ldots,q-1\}$ be arbitrary and assume $R\cap (q\N-r)\neq\emptyset$. Fix any $x\in R\cap (q\N-r)$. Let $0<\epsilon<\frac{1}{q}$ and choose $m\in\N$ such that for each $s\in \{0,1,\ldots,m-1\}$ the intersection $R\cap (m\N+s)$ is either empty or has lower density $>\frac{1-\epsilon}{m}$. Take $s\in \{0,1,\ldots,m-1\}$ such that $s\equiv x\bmod m$. Since $x\in R$ and $x\in m\N+s$, the intersection $R\cap (m\N+s)$ is non-empty and hence $d(R\cap (m\N+s))>\frac{1-\epsilon}{m}$. On the other hand, $d((q\N+r)\cap (m\N+s))\geq\frac{1}{qm}$. It follows that $d(R\cap (m\N+s)\cap (q\N+r))>\frac{1}{m}(\frac{1}{q}-\epsilon)>0$. This finishes the proof of the claim.
\end{Remark}

We need two lemmas for the proof of \cref{prop:Drat} which we state next.

\begin{Lemma}[see {\cite[Lemma 2.7]{BR02}} applied to $B=\{p^2:p\in P\}\cup(\P\setminus P)$]
\label{lem:-AB}
Let $P\subset\P$ with $\sum_{p\in\P\setminus P}\tfrac{1}{p}<\infty$, and let $S_P$ be the set defined in formula \eqref{eqn:S_P}.
Then $S_P$ is inner regular. In particular, according to \cref{rem:inner-reg-AB}, $S_P$ satisfies conditions \ref{itm:Drat-1} and \ref{itm:Drat-2}.
\end{Lemma}

\begin{Lemma}
\label{lem:lifting-trick-7}
Let $P\subset\P$ with $\sum_{p\in\P\setminus P}\tfrac{1}{p}<\infty$, let $f$ be multiplicative function. Let $S_P$ be the set defined in formula \eqref{eqn:S_P} and for $t\in\N$ let $S_P^{(t)}:=\{s\in S_P:\gcd(s,t)=1\}$.
If for all $t\in\N$ and $z\in\C$ the set $E(f,z)\cap S_P^{(t)}$ satisfies \ref{itm:Drat-1} and \ref{itm:Drat-2} then for all $z\in\C$ the set $E(f,z)$ satisfies \ref{itm:Drat-1} and \ref{itm:Drat-2}.
\end{Lemma}

\begin{proof}
Let $T_P$ be defined as
\begin{equation}\label{eqn:T_P}
T_P:=\left\{n\in\N:\text{for all $p\in P$ if $p\mid n$ then $p^2\mid n$}\right\}.
\end{equation}
Since any natural number $n$ can be written uniquely as $st$, where $s\in S_P$, $t\in T_P$ and $\gcd(s,t)=1$, $\N$ can be partitioned into
\begin{equation}
\label{eq:S_P-T_P-partition}
\N=\bigcup_{t\in T_P}tS_P^{(t)}.
\end{equation}
Note that $d(S_P)=M(\1_{S_P})$ exists (due to \cref{thm:wirsing}) and $d(S_P)>0$ because $\sum_{p\in\P\setminus P}\tfrac{1}{p}<\infty$ and therefore
$\sum_{p\in\P}\frac1p\big(1-\1_{S_P}(p)\big)<\infty$ (cf.\ \cref{lem:lH2-C}).
Likewise, $\1_{S_P^{(t)}}$ is a multiplicative function and hence $d(S_P^{(t)})=M(\1_{S_P}^{(t)})$ exists (again due to \cref{thm:wirsing}) and is positive (also by \cref{lem:lH2-C}). Using~\eqref{eq:S_P-T_P-partition} and the fact that $d(tS_P^{(t)})=t^{-1}d(S_P^{(t)})$ we obtain
\begin{equation}
\label{eq:S_P-T_P-1}
\sum_{t\in T_{P}}\tfrac{d(S_P^{(t)})}{t}=\sum_{t\in T_{P}}d(tS_P^{(t)})\leq d\left(\bigcup_{t\in T_{P}} tS_P^{(t)}\right) =d(\N)=1.
\end{equation}

For $t\in T_P$ let $u_t:=f(t)$ and define
\begin{equation}
\label{eq:e_t-1}
E_t:=
\begin{cases}
E(f, \tfrac{z}{u_t}),&\text{if}~u_t\neq 0;
\\
\emptyset,&\text{if $u_t=0$ and $z\neq 0$};
\\
\N,&\text{if $u_t=0$ and $z= 0$}.
\end{cases}
\end{equation}
It is easy to check that $E(f,z)\cap tS_P^{(t)}=t(E_t\cap S_P^{(t)})$.
Observe that if $E_t=E(f, \tfrac{z}{u_t})$ then $E_t\cap S_P^{(t)}$ satisfies \ref{itm:Drat-1} and \ref{itm:Drat-2} due to the assumptions stipulated in the statement of \cref{lem:lifting-trick-7}.
Also, if $E_t=\emptyset$ then $E_t\cap S_P^{(t)}=\emptyset$ obviously satisfies \ref{itm:Drat-1} and \ref{itm:Drat-2}. In light of \cref{lem:-AB}, if $E_t=\N$ then $E_t\cap S_P^{(t)}=S_P^{(t)}$ satisfies \ref{itm:Drat-1} and \ref{itm:Drat-2}.
We see that for each of the three cases comprising the definition of $E_t$ in \eqref{eq:e_t-1}, the set $E_t\cap S_P^{(t)}$ satisfies \ref{itm:Drat-1} and \ref{itm:Drat-2}.

Since $E_t\cap S_P^{(t)}$ satisfies \ref{itm:Drat-1} and \ref{itm:Drat-2} and $E(f,z)\cap tS_P^{(t)}=t(E_t\cap S_P^{(t)})$, it follows that $E(f,z)\cap tS_P^{(t)}$ satisfies \ref{itm:Drat-1} and \ref{itm:Drat-2}.

Note that any finite union of sets satisfying \ref{itm:Drat-1} and \ref{itm:Drat-2} also satisfies \ref{itm:Drat-1} and \ref{itm:Drat-2}. Therefore, for every $M\geq 1$, the set
$$
B_M:=\bigcup_{t\in T_P\atop t\leq M}\left(E(f,z)\cap tS_P^{(t)}\right)
$$
satisfies \ref{itm:Drat-1} and \ref{itm:Drat-2}.
Finally, since $d(E(f,z)\setminus B_M)=0$ as $M\to\infty$ (see equation \eqref{eq:S_P-T_P-1}), we conclude that $E(f,z)$ satisfies \ref{itm:Drat-1}. Since  $B_1\subset B_2\subset\ldots$ and $E(f,z)=\bigcup_{M\geq 1}B_M$, we conclude that $E(f,z)$ satisfies \ref{itm:Drat-2}. This finishes the proof.
\end{proof}

\begin{proof}[Proof of \cref{prop:Drat}]
Let $R\in\Drat$ with $d(R)>0$ be given. Then there exist a Besicovitch rationally almost periodic multiplicative function $f$ and a complex number $z$ such that $R=E(f,z)$.
Note that if $z=0$ then it follows from \cref{lem:ST-zeros-of-f-form-rational-set} that $R$ satisfies \ref{itm:Drat-1} and \ref{itm:Drat-2}. We can therefore assume without loss of generality that $z\neq 0$.

We now apply \cref{cor:density->concentrated function} to find a concentrated multiplicative function $g\colon \N\to\C\setminus\{0\}$ such that the set $P':=\{p\in\P:f(p)\neq g(p)\}$ satisfies $\sum_{p\in P'}\frac{1}{p}<\infty$.
Since $f$ is Besicovitch rationally almost periodic, it follows from \cref{cor:DD82-thm1-and-thm6} that there exists a Dirichlet character $\chi$ such that
$\sum_{p\in\P}\frac{1}{p}(1-f(p)\overline{\chi(p)})$ converges. In particular, $\mathbb{D}(f,\chi)<\infty$.

The function $g$ is a concentrated multiplicative function and therefore its concentration group $\G$ is a finite set of roots of unity and we have $\sum_{\substack{p\in\P,\\ g(p)\notin\G}}\frac{1}{p}<\infty$.
Define
$$
P'':=\{p\in\P:f(p)\neq\chi(p)\}
$$
and let $\rho:=\min\{1-\text{Re}(x\overline{y}): x\in\G,~y\in im(\chi),~x\neq y\}$. Note that $\rho>0$ and
\begin{eqnarray*}
\sum_{p\in P''}\frac{1}{p}
&\leq &
\sum_{\substack{p\in\P,\\ g(p)\notin\G}}\frac{1}{p}~+~\sum_{p\in P'}\frac{1}{p}
~+~
\sum_{\substack{p\in\P,\\ f(p)\neq \chi(p),\\ g(p)\in\G,\\g(p)=f(p)}}\frac{1}{p}
\\
&\leq &
\sum_{\substack{p\in\P,\\ f(p)\notin\G}}\frac{1}{p}~+~\sum_{p\in P'}\frac{1}{p}~+~\frac{1}{\rho}\mathbb{D}(f,\chi)~<~\infty.
\end{eqnarray*}
Let $P:=\P\setminus P''$, let $S_P$ be the set defined in formula \eqref{eqn:S_P} and, for $t\in\N$, let $S_P^{(t)}:=\{s\in S_P:\gcd(s,t)=1\}$.
Since $f(p)=\chi(p)$ for all $p\in P$, we conclude that $E(f,z)\cap S_P^{(t)}=E(\chi,z)\cap S_P^{(t)}$. Recall that all Dirichlet characters are periodic functions. Therefore the set $E(\chi,z)$ is either empty or a finite union of infinite arithmetic progressions. In view of \cref{lem:-AB}, the set $S_P^{(t)}$ is inner regular. Hence $E(\chi,z)\cap S_P^{(t)}$ is an inner regular set. Therefore, by \cref{rem:inner-reg-AB}, for all $t\in\N$ the set $E(f,z)\cap S_P^{(t)}=E(\chi,z)\cap S_P^{(t)}$ satisfies \ref{itm:Drat-1} and \ref{itm:Drat-2}. Finally, we can apply \cref{lem:lifting-trick-7} and conclude that $E(f,z)$ satisfies \ref{itm:Drat-1} and \ref{itm:Drat-2}.
\end{proof}

\cref{prop:Drat} immediately gives the following corollary.

\begin{Corollary}
\label{cor:D_rat}
Let $R\in\Drat$ with $d(R)>0$. Then for all $r\in R$ the set $R-r$ is divisible (cf.\ \cref{def_divisibility-property}). 
\end{Corollary}

\subsection{Proofs of \cref{thm:recurrence-divisibility} and \cref{thm:recurrence-enhanced}}
\label{subsec:convergence}

\begin{proof}[Proof of \cref{thm:recurrence-enhanced}]
Suppose $E\in\Dzero$ has positive density,  $R\in\Drat$ and $E$ is uniform relative to $R$. Our goal is to show that for all $r\in R$ the set $E-r$ is divisible.

It follows from \cref{prop:Drat} and \cref{cor:D_rat} that for all $r\in R$ and $q\in\N$ the density $d((R-r)\cap q\N)=d(R\cap (q\N+r))$ exists and is positive. However, since the function $d\left(R\right)\1_{E}-d\left(E\right)\1_{R}$ is uniform, it follows from \cref{prop:interplay-uniform-almost-periodic}, part \ref{itm:prop:interplay-uniform-almost-periodic-b0}, that
$d\left(R\right)\1_{E\cap (q\N+r)}-d\left(E\right)\1_{R\cap(q\N+r)}$ is uniform. Since all uniform functions have zero mean, we deduce that $d(E\cap (q\N+r))$ also exists and that
$$
d\left(R\right)d(E\cap (q\N+r))-d\left(E\right)d(R\cap(q\N+r))=0.
$$
Thus, it follows from $d(R\cap (q\N+r))>0$ that $d(E\cap (q\N+r))>0$. This proves that $E-r$ is divisible.
\end{proof}

\begin{Example}\label{eg:counterexample-selfshifts}
Consider the multiplicative function
$$
f(n):=
\begin{cases}
1&\text{if}~n=2^k m,~\text{where}~k\in\{0,2,4,6,\ldots\}~\text{and}~2\nmid m.
\\
0&\text{otherwise.}
\end{cases}
$$
Clearly, $f$ is rationally Besicovitch almost periodic (see \cref{cor:DD82-thm1-and-thm6}) and therefore the level set $E=E(f,1)=\{n\in\N:f(n)=1\}$ belongs to $\Drat$.

Note that $E-r$ is divisible for all $r\in\N\cup\{0\}$. This should be juxtaposed with the fact that for the set of squarefree numbers $Q$ one has that $Q-r$ is divisible if and only if $r\in Q$.
\end{Example}

Next, we embark on the proof of \cref{thm:recurrence-divisibility}.  We will need the following two results.

\begin{Theorem}[\cite{BKLR16arXiv}]
\label{thm:BKLR-4.5}
Let $R\subset \N$ be rational and let $r\in\N\cup\{0\}$. Then the following are equivalent:
\begin{itemize}
\item
$R-r$ is divisible;
\item
$R-r$ is an averaging set of recurrence;
\item
$R-r$  is an averaging set of polynomial multiple recurrence.
\end{itemize}
\end{Theorem}

\begin{Lemma}[Lemma 3.5, \cite{FHK13}]\label{l11}
Let $\xbmt$ be an invertible measure preserving system, $k\in\N$,
$p_1,\ldots,p_\ell \in \Z[x]$, $f_1,\ldots,f_\ell\in L^\infty\xbm$ bounded by $1$ and let $F\colon\N\rightarrow\C$ be bounded by $1$ as well. Then there exists an integer $s\in\N$, that only depends on $k$ and the maximal degree of the polynomials $p_1,\ldots,p_\ell$, such that
\[
\left\|\frac{1}{N}\sum_{n=1}^N F(n) T^{p_1(n)}f_1\cdots T^{p_\ell(n)}f_\ell \right\|_{L^2\xbm}~=~\Oh\left(\|F\|_{U^s_{[N]}}\right)+\oh(1).
\]
\end{Lemma}

\begin{proof}[Proof of \cref{thm:recurrence-divisibility}]
Let $E\in\Dzero$ and $r\in\N\cup\{0\}$. It suffices to show that if $E-r$ is divisible then $E-r$ is an averaging set of polynomial multiple recurrence, since all the other implications formulated in \cref{thm:recurrence-divisibility} are obvious. 

Thus, assume $E-r$ is divisible. Note that by \cref{thm:structure-theorem} there exists $R\in\Drat$ such that $E$ is uniform relative to $R$. According to \cref{prop:Drat}, the set $R$ is rational. Moreover, it follows from $E-r\subset R-r$ that $R-r$ is divisible.

Let $\xbmt$ be an arbitrary invertible measure preserving system, let
$A\in\mathcal{B}$ with $\mu(A)>0$ and let $p_i\in\Z[x]$, $i=1,\ldots,\ell$
with $p_i(0)=0$ be given.
Using \cref{l11} and the fact that $d\left(R\right)\1_{E}-d\left(E\right)\1_{R}$ is uniform, we get that the limit
\begin{equation}
\label{eqn-fi-0}
\lim_{N\to\infty}\frac{1}{N}\sum_{n=1}^N\1_{E-r}(n)
 \mu\big(A\cap T^{-p_1(n)}A\cap\ldots\cap T^{-p_\ell(n)}A\big)
\end{equation}
is the same as the limit
\begin{equation}
\label{eqn-fi-1}
 \frac{d(E)}{d(R)}\lim_{N\to\infty}\frac{1}{N}\sum_{n=1}^N \1_{R-r}(n)
\mu\big(A\cap T^{-p_1(n)}A\cap\ldots\cap T^{-p_\ell(n)}A\big),
\end{equation}
(meaning that the limit in \eqref{eqn-fi-0} exists if and only if the limit in \eqref{eqn-fi-1} exists and then they are equal).
Using \cref{thm:BKLR-4.5} and the fact that $R$ is rational and $R-r$ is divisible, we conclude that the limit in \eqref{eqn-fi-1} exists and is positive. It follows that the limit in \eqref{eqn-fi-0} exists and is positive. Hence $E-r$ is an averaging set of polynomial multiple recurrence.
\end{proof}

\begin{Remark}
It is natural to ask whether \cref{thm:recurrence-divisibility} can be extended to a more general setting involving several commuting measure preserving transformations. The methods used in this section to derive a proof \cref{thm:recurrence-divisibility} are general enough to also work in this more general set-up, except one missing ingredient, which is a version of \cref{thm:BKLR-4.5} for several commuting transformations. This extension of \cref{thm:BKLR-4.5} is stated as an open problem (Question 2.10) in \cite{BKLR16arXiv}. Informally, the problem boils down to removing the $\epsilon$ in Theorem 1.1 in \cite{Frantzikinakis15}.
\end{Remark}

\subsection{Level sets of multiplicative functions are good for averaging convergence}

A set $E\subset\N$ is an \define{averaging set of polynomial multiple convergence} if for all invertible measure preserving systems $\xbmt$, all $\ell\geq 1$, all $f_1,\ldots,f_\ell\in L^\infty\xbm$ and all polynomials
$p_i\in\Z[x]$, $i=1,\ldots,\ell$, the limit
\begin{equation*}
\lim_{N\to\infty}\frac{1}{N}\sum_{n=1}^N1_E(n)
T^{p_1(n)}f_1\cdots T^{p_\ell(n)}f_\ell
\end{equation*}
exists in $L^2\xbm$.

In this subsection we give a proof of the following theorem.

\begin{Theorem}
Suppose $E\in\Dzero$ has positive density. Then $E$ is an averaging set of polynomial multiple convergence.
\end{Theorem}

\begin{proof}
Let $\xbmt$ be an arbitrary invertible measure preserving system, $\ell\geq 1$, $f_1,\ldots,f_\ell\in L^\infty\xbm$ and $p_1,\ldots,p_\ell\in\Z[x]$.

By \cref{thm:structure-theorem} we can find $R\in\Drat$ such that $E$ is uniform relative to $R$. Then, in light of \cref{l11} and the fact that $d\left(R\right)\1_{E}-d\left(E\right)\1_{R}$ is uniform, we have
\begin{equation*}
\lim_{N\to\infty}\left\|
\frac{1}{N}\sum_{n=1}^N\1_{E}(n)
 T^{p_1(n)}f_1\cdots T^{p_\ell(n)}f_\ell -
 \frac{d(E)}{d(R)}\frac{1}{N}\sum_{n=1}^N \1_{R}(n)
T^{p_1(n)}f_1\cdots T^{p_\ell(n)}f_\ell\right\|_{L^2\xbm}=0.
\end{equation*}
It was shown in \cite[Section 2]{BKLR16arXiv} that any rational set of positive density is an averaging set of polynomial multiple convergence. Therefore
\begin{equation*}
\lim_{N\to\infty}\frac{1}{N}\sum_{n=1}^N \1_{R}(n)
T^{p_1(n)}f_1\cdots T^{p_\ell(n)}f_\ell
\end{equation*}
exists in $L^2\xbm$. This proves that
\begin{equation*}
\lim_{N\to\infty}\frac{1}{N}\sum_{n=1}^N \1_{E}(n)
T^{p_1(n)}f_1\cdots T^{p_\ell(n)}f_\ell
\end{equation*}
exists and hence $E$ is an averaging set of polynomial multiple convergence.
\end{proof}

\appendix
\counterwithin{Theorem}{section}

\section{Appendix}

\begin{proof}[Proof of \cref{lem:mconv-characterization}]
For $f\in\mconv$ and for any periodic function $g\colon \N\to\C$, say of period $q$, the mean $M(f\cdot g)$ exists, because $M(f\cdot g)=\sum_{r=0}^{q-1}g(r)\lim_{N\to\infty}\frac1N\sum_{n=1}^N f(qn+r)$. Since any Dirichlet character $\chi$ is periodic, we conclude that $M(\chi\cdot f)$ always exists.

It thus remains to show that for any multiplicative function $f$ bounded by $1$ for which $M(\chi\cdot f)$ exists for all Dirichlet characters $\chi$, we have
\begin{equation}\label{eq:sup-inf}
\lim_{N\to\infty}\frac{1}{N}\sum_{n=1}^N f(qn+r)\quad\text{exists for all $q,r\in\N$.}
\end{equation}
This is equivalent to the assertion that
\begin{equation}\label{eq:sup-inf-alt}
\lim_{N\to\infty}\frac{1}{N}\sum_{n=1}^N \1_{q\Z+r}(n)f(n)\quad\text{exists for all $q,r\in\N$.}
\end{equation}

We prove \eqref{eq:sup-inf} by induction on $d=\gcd(q,r)$.
The beginning of the induction is given by $d=1$. In this case the numbers $q$ and $r$ are coprime, which implies that the function $\1_{q\Z+r}$ can be written as a finite linear combination of Dirichlet characters in the following way:
$$
\1_{q\Z+r}(n)=\frac1{\tot(q)}\sum_{\chi\bmod q}\ov{\chi}(r)\chi(n),
$$
where $\sum_{\chi~{\rm mod}~q}$ denotes the sum over all Dirichlet characters of modulus $q$. 
Therefore,
$$
\lim_{N\to\infty}\frac{1}{N}\sum_{n=1}^N \1_{q\Z+r}(n)f(n)
=
\frac1{\tot(q)}\sum_{\chi~{\rm mod}~q}\ov{\chi}(r)\lim_{N\to\infty}\frac1N\sum_{n=1 }^N f(n)\chi(n).
$$
From this \eqref{eq:sup-inf-alt}, and therefore also \eqref{eq:sup-inf}, follows immediately.

Next, we prove the inductive step.
Let $d_0>1$ and assume that \eqref{eq:sup-inf} has already been proven for all pairs $q$ and $r$ with $\gcd(q,r)<d_0$.
We will show \eqref{eq:sup-inf} for all pairs $q$ and $r$ with $\gcd(q,r)=d_0$.

In the following, for a set $D\subset \C$ of complex numbers we will use $\acc{\,D}$ to denote the set of accumulation points of $D$. Note that for a bounded sequence $(x_N)_{N\in\N}$ of complex numbers the limit $\lim_{N\to\infty}x_N$ exists if and only if
$$
\diam{\acc\{x_N:N\in\N\}}=0,
$$
where $\diam{\cdot}$ is used to denote the diameter of a set.
We make the following claim:

\noindent\textit{\underline{Claim:}} Let $p\in\P$ be an arbitrary prime number and let $d\in\N$ be a natural number satisfying $p\nmid d$ and $d<d_0$.
Then for all pairs $Q,R\in\N$ for which $\gcd(Q,R)=dp^k$ for some $k\in\N$, one of two possibilities holds:
\begin{enumerate}
[label=\text{(\alph*)}, ref=\text{(\alph*)}, leftmargin=*]
\item\label{itm:ls-li-a}
Either
\begin{equation}\label{eq:sup-inf-K-a}
\diam{\acc{\left\{\frac{1}{N}\sum_{n=1}^N f(Qn+R) : N\in\N\right\}}}=0,
\end{equation}
\item\label{itm:ls-li-b}
or there exist $Q',R'\in\N$ such that $\gcd(Q',R')=dp^{k+1}$ and
\begin{equation}\label{eq:sup-inf-K-b}
\begin{split}
\diam{\acc{\left\{\frac{1}{N}\sum_{n=1}^N f(Qn+R)
: N\in\N\right\}}}&\\
\leq\frac{1}{p}~\diam{\acc{\left\{
\frac{1}{N}\sum_{n=1}^N f(Q'n+R'): N\in\N\right\}}}&.
\end{split}
\end{equation}
\end{enumerate}

Before we proceed to prove this claim, let us see how we can use it to finish the proof of the inductive step. Hence, let $q,r\in\N$ with $\gcd(q,r)=d_0$.
Since $d_0>1$, we can find $p\in\P$ and $k\in\N$ such that $d_0=dp^k$ and $\gcd(d,p)=1$. Observe that $d<d_0$. To prove that the limit of $\frac{1}{N}\sum_{n=1}^N f(qn+r)$ exists as $N\to\infty$, it suffices to show that for all $\epsilon>0$ one has
\begin{equation}
\label{eqn:ls-li-4}
\diam{\acc{\left\{\frac{1}{N}\sum_{n=1}^N f(qn+r) : N\in\N\right\}}}<\epsilon.
\end{equation}
Thus, let $\epsilon>0$ be arbitrary. We apply the above claim and find ourselves either in case \ref{itm:ls-li-a} or in case \ref{itm:ls-li-b}. If we end up in case \ref{itm:ls-li-a} then \eqref{eqn:ls-li-4} holds and we are done. If we are in case \ref{itm:ls-li-b} then we obtain a new pair of numbers $q',r'\in\N$ with $\gcd(q',r')=dp^{k+1}$. We then apply the claim again to this new pair of numbers $q'$ and $r'$.
We continue this procedure and, after $j$-many applications of the claim, it follows either from
\eqref{eq:sup-inf-K-a} that
$$
\diam{\acc{\left\{\frac{1}{N}\sum_{n=1}^N f(qn+r) : N\in\N\right\}}}=0,
$$
or from \eqref{eq:sup-inf-K-b} that
$$
\diam{\acc{\left\{\frac{1}{N}\sum_{n=1}^N f(qn+r) : N\in\N\right\}}}\leq \frac{1}{p^j}.
$$
If $j$ is sufficiently large, then $\frac{1}{p^j}<\epsilon$ and hence \eqref{eqn:ls-li-4} is proven.

It remains to prove the above claim. Let $p\in\P$. Assume that $k,d\in\N$ with $p\nmid d$ and $d<d_0$ and let $Q,R\in\N$ satisfy $\gcd(Q,R)=dp^k$.
Define $Q_0:=Qp^{-k}$ and $R_0:=Rp^{-k}$.
We now distinguish two cases, the case $p\mid Q_0$ (which will correspond to part \ref{itm:ls-li-a} of the claim) and the case $p\nmid Q_0$ (which will correspond to part \ref{itm:ls-li-b} of the claim).
In the first case we have $p\nmid R_0$, because otherwise we have $p^{k+1}\mid \gcd(Q,R)$ which contradicts $p\nmid d$. Therefore the equation 
\begin{equation}
\label{eq:crt-1}
Q_0 x+R_0\equiv 0\bmod p
\end{equation}
has no solutions in $x$. This implies that for any $n\in\N$ the number $Q_0 n+R_0$ is coprime to $p^k$.
Hence
\begin{equation}\label{eq:sup-inf-5}
\frac{1}{N}\sum_{n=1}^N f(Qn+R)
=f(p^k)\left(\frac{1}{N}\sum_{n=1}^Nf(Q_0n+R_0)\right).
\end{equation}
However, we have that $\gcd(Q_0,R_0)=d$ and $d<d_0$. Therefore, by the induction hypothesis, the limit of \eqref{eq:sup-inf-5} as $N\to\infty$ exists and so \eqref{eq:sup-inf-K-a} is satisfied.

Next, assume $p\nmid Q_0$. In this case equation \eqref{eq:crt-1} possesses exactly one solution for $x\in \{0,\ldots,p-1\}$ which we denote by $x_0$. We deduce that $p^k$ is coprime to $Q_0n+R_0$ if and only if $n\not\equiv x_0\bmod p$. In particular, we have that $f(Qn+R)=f(p^k)f(Q_0n+R_0)$ for all $n\in\N$ with $n\not\equiv x_0\bmod p$.
Define $Q':=pQ$ and $R':=Qx_0+R$.
We obtain
\begin{eqnarray*}
\frac{1}{pN}\sum_{n=1}^{pN} f(Qn+R)
&=&\frac{1}{pN}\sum_{n=1}^{pN} f(p^k(Q_0n+R_0))
\\
&=&\frac{1}{pN}\sum_{x=0}^{p-1}\sum_{n=1}^N f(p^k(Q_0(pn+x)+R_0))
\\
&=&\sum_{x\in\{0,1\ldots,p-1\}\setminus\{x_0\}} \frac{f(p^k)}{p}\left(\frac{1}{N}\sum_{n=1}^N f(Q_0pn+Q_0x+R_0) \right)
\\
&&\qquad\qquad+\frac{1}{p}\left(\frac{1}{N}\sum_{n=1}^N f(Q'n+R')\right).
\end{eqnarray*}
Define $Q_1:=Q_0p$ and $R_x:=Q_0x+R_0$. Then,
$$
\frac{1}{N}\sum_{n=1}^N f(Q_0pn+Q_0x+R_0)
=\frac{1}{N}\sum_{n=1}^N f(Q_1n+R_x).
$$
Furthermore, for $x\neq x_0$, we have $\gcd(Q_0,Q_0x+R_0)=\gcd(Q_0,R_0)=d$ and therefore
$\gcd(Q_1,R_x)$ is either equal to $d$ or to $dp$. However, $\gcd(Q_1,R_x)$ cannot be equal to $dp$ because $R_x\not\equiv 0\bmod p$ for $x\neq x_0$. Hence $\gcd(Q_1,R_x)=d$. Since $d<d_0$, we can use the induction hypothesis to deduce that the limit of $\frac{1}{N}\sum_{n=1}^N f(Q_1n+R_x)$ exists as $N\to\infty$ for all $x\neq x_0$. Therefore,
\begin{equation*}
\begin{split}
&\diam{\acc{\left\{\frac{1}{N}\sum_{n=1}^N f(Qn+R)
: N\in\N\right\}}}\\
&=\diam{\acc{\left\{\frac{1}{pN}\sum_{n=1}^{pN} f(Qn+R)
: N\in\N\right\}}}\\
&=\diam{\acc{\left\{\frac{1}{pN}\sum_{n=1}^{N} f(Q'n+R')
: N\in\N\right\}}}\\
&=\frac{1}{p}~\diam{\acc{\left\{\frac{1}{N}\sum_{n=1}^{N} f(Q'n+R')
: N\in\N\right\}}}.
\end{split}
\end{equation*}
Moreover, since $\gcd(Q_0,Q_0x_0+R_0)=\gcd(Q_0,R_0)=d$ and since $Q_0x_0+R_0\equiv 0\bmod p$, we have that $\gcd(Q',R')=dp^{k+1}$. This shows that we are in case \ref{itm:ls-li-b} of the claim.
\end{proof}

\bibliographystyle{siam}

\providecommand{\noopsort}[1]{} 

\allowdisplaybreaks
\small
\bibliography{BibMF}

\bigskip
\footnotesize
\noindent
Vitaly Bergelson\\
\textsc{Department of Mathematics, Ohio State University, Columbus, OH 43210, USA}\par\nopagebreak
\noindent
\textit{E-mail address:}
\href{mailto:vitaly@math.ohio-state.edu}
{\texttt{vitaly@math.ohio-state.edu}}

\medskip

\noindent
Joanna\ Ku\l aga-Przymus\\
\textsc{Aix-Marseille Universit\'e, Centrale Marseille, CNRS, Institut de Math\'ematiques de Marseille, UMR7373, 39 Rue F.\ Joliot Curie 13453, Marseille, France}\\
\textsc{Faculty of Mathematics and Computer Science, Nicolaus Copernicus University, Chopina 12/18, 87-100 Toru\'{n}, Poland}\par\nopagebreak
\noindent
\textit{E-mail address:}
\href{mailto:joanna.kulaga@gmail.com}
{\texttt{joanna.kulaga@gmail.com}}

\medskip

\noindent
Mariusz Lema\'nczyk\\
\textsc{Faculty of Mathematics and Computer Science, Nicolaus Copernicus University, Chopina 12/18, 87-100 Toru\'{n}, Poland}\par\nopagebreak
\noindent
\textit{E-mail address:} 
\href{mailto:mlem@mat.umk.pl}
{\texttt{mlem@mat.umk.pl}}

\medskip

\noindent
Florian K.\ Richter\\
\textsc{Department of Mathematics, Ohio State University, Columbus, OH 43210, USA}\par\nopagebreak
\noindent
\textit{E-mail address:}
\href{mailto:richter.109@osu.edu}
{\texttt{richter.109@osu.edu}}

\end{document}